\documentclass[12pt,a4paper]{article}
\usepackage{amsmath}
\usepackage{amssymb}
\usepackage{graphicx}
\usepackage{color}
\usepackage{amsthm}
\usepackage{amsfonts}
\usepackage{bbm}

\allowdisplaybreaks

\numberwithin{equation}{section}
\title{One-dimensional random interlacements}
\author{Darcy Camargo \and Serguei Popov}
\newcommand{\pr}{\mathbb{P}}
\newcommand{\ord}{\mathcal{O}}
\newcommand{\tori}{\mathbb{Z}/n\mathbb{Z}}

\newcommand{\vac}{\mathcal{V}^\alpha}
\newcommand{\cpr}{\widetilde{\mathbb{P}}}
\newcommand{\rpr}{\widehat{\mathbb{P}}}
\newcommand{\ex}{\mathbb{E}}
\newcommand{\rex}{\mathbb{\widehat{E}}}
\newcommand{\ind}{\mathbbm{1}}
\newcommand{\zp}{\mathbb{Z}^+}
\newcommand{\zd}{\mathbb{Z}^d}

\newcommand\lawlimit{\mathrel{\overset{\makebox[0pt]{\mbox{\tiny law}}}{\Rightarrow}}}
\DeclareMathOperator{\cp}{Cap}
\DeclareMathOperator{\cpt}{\widehat{Cap}}
\DeclareMathOperator{\diam}{Diam}

\DeclareMathOperator{\hm}{hm}
\newtheorem{teo}{Theorem}[section]
\newtheorem{lemma}[teo]{Lemma}

\newtheorem{coro}[teo]{Corollary}

\let\phi\varphi
\let\epsilon\varepsilon
\begin{document}
\bibliographystyle{plain}
\maketitle

{\footnotesize  \noindent Department of Statistics, Institute of Mathematics,
 Statistics and Scientific Computation, University of Campinas --
UNICAMP, rua S\'ergio Buarque de Holanda 651,
13083--859, Campinas SP, Brazil\\
\noindent e-mails: \texttt{darcygcamargo@gmail.com, popov@ime.unicamp.br}

}
\begin{abstract}
We base ourselves on the construction of the two-dimensional random interlacements~\cite{2dint} to define the one-dimensional version of the process. For this constructions we consider simple random walks conditioned on never hitting the origin, which makes them transient. We also compare this process to the conditional random walk on the ring graph. Our results are the convergence of the vacant set on the ring graph to the vacant set of one-dimensional random interlacements, a central limit theorem for the interlacements' local time for sites far from the origin and the convergence in law of the local times of the conditional walk on the ring graph to the interlacements' local times.
\\[.3cm]\textbf{Keywords:} random interlacements, local times, occupation times, simple random walk, Doob's $h$-transform. 
\\[.3cm]\textbf{AMS 2000 subject classifications:}  60G50, 60K35, 60F05.
\end{abstract}

\section{Introduction and results}

The process of random interlacements was initially introduced by Alain Sol Sznitman in \cite{inter1}. It can be viewed as a Poisson point process of rate $\alpha$ (what is also the process parameter) in a space of doubly infinite trajectories of simple random walks in dimension $d\geq3$. The law of the \emph{vacant set} (i.e., the set of unvisited sites) is completely characterized by
\[
\pr[A\subset \vac]=e^{-\alpha\cp(A)}\mbox{,\quad for all $A\subset \zd$},
 \] 
where $\cp$ stands for the classical capacity for random walks 
(see section~6.5 of~\cite{rwb}). First results about this process were mainly about the non-triviality of the vacant set percolation and its properties~\cite{inter1,inter2, inter3} and also how this process can be obtained as a local picture of the trace left by the random walk in the $d$-dimensional torus (see e.g.~\cite{torusint}). A good introduction to the subject can be found in~\cite{ribook} and in \cite{randomint}.

 In \cite{2dint} the model of random interlacements in two dimensions was introduced. This process could not be defined using the classical approach,  as the simple random walk in two dimensions is recurrent and so just one trajectory would cover the entire discrete plane $\mathbb{Z}^2$, leaving nothing to be seen. Therefore, in order to construct the process, one uses simple random walks conditioned to never hitting the origin. This conditioning turns the walk transient and the construction of the process becomes possible, 
at cost of loosing the stationarity (there is, however, a so-called
\emph{conditional stationarity}, see Theorem~2.3~(i) of~\cite{2dint}). In its construction a parameter change was made to make the formulas cleaner, so the law of the vacant set is characterized by
\[
\pr[A\subset \vac]=e^{-\alpha\pi\cp(A\cup\{0\})}.
 \] 
In this article we base ourselves on the approach of~\cite{2dint} to construct the one-dimensional random interlacements process. We were able to define the process making use of conditional random walks in this construction. It was defined in such a way that holds for any $A\subset\mathbb{Z}$ that
\[
\pr[A\subset \vac]=e^{-\alpha\cp(A\cup\{0\})}=e^{-\alpha\diam(A)/2},
\]
where $\diam(A)$ stands for the diameter of the set.

  Here percolation questions are not of interest, since our vacant set is an interval containing the origin; so, we focus on other problems, mainly about the relation to the ring graph (the one-dimensional torus) and the local times of the process.
 
 A well studied problem about random interlacements is how it represents the local picture of a random walk on a torus, when it is left to run for a certain fixed time. Consider the $d$-dimensional torus $(\tori)^d$ and denote the trace left for a random walk until time $t$ for $X_{[0,t]}$. In Theorem~1.1 of~\cite{torusint} it was shown that for any $\alpha>0$, $\delta>0$ and $\epsilon \in (0,1)$ we can construct a coupling between the random interlacements and the random walk on this torus in such a way that for a constant $c$ depending on $\alpha$, $\delta$, $\epsilon$ we have
 \begin{equation}
 \pr[\mathcal{I}^{\alpha(1-\epsilon)}\cap A\subseteq X_{[0,\lfloor \alpha n^d \rfloor]}\cap A\subseteq\mathcal{I}^{\alpha(1+\epsilon)}\cap A]\geq 1-c n^{-\delta}.
  \end{equation}
Denoting the vacant set left by the random walk on the torus by $V_t^d=(\tori)^d\backslash X_{[0,t]}$ for $d\geq 3$, it satisfies:
\begin{equation*}
V_{\lfloor \alpha n^d\rfloor}^d\lawlimit \vac_d,
\end{equation*}
where $\vac_d$ stands for the vacant set for the random interlacements in $\zd$.

In Theorem~2.6 of~\cite{2dint} this result was extended to the random interlacements in two dimensions and a simple random walk on the torus conditioned to never hitting the origin. For $d=2$ we use an analogous notation for the multidimensional version of vacant set, but here we have conditional random walks. So, let $\{\hat{X}_t\}_{t\geq 0}$ be the random walk on the two-dimensional torus conditioned to never hitting the origin and $\hat{X}_{[0,t]}$ its trajectory until time $t$. Denoting $V_t=\mathbb{Z}^2\backslash \hat{X}_{[0,t]}$,
 Theorem~2.6 of~\cite{2dint} states that
\begin{equation}
V_{\left\lfloor \frac{4\alpha}{\pi} n^2\ln^2 n\right\rfloor}^2\lawlimit \vac_2.
\end{equation}

The first result we prove is the one-dimensional random interlacements version of this theorem. The one-dimensional discrete torus with $n$ sites is in fact the ring graph $R_n$. Here we will consider a simple random walk conditioned to never hitting the origin. 
\begin{teo}
\label{teolimit}
Let $X_t$ be the conditional random walk on $R_{n}$ started at $n$ and $V_t=\{x\in R_{n}\mid X_n\neq x \mbox{ for }n\leq t\}$, then
\[
V_{\left\lfloor\frac{\alpha n^3}{2\pi^2}\right\rfloor}\lawlimit \mathcal{V}^\alpha.
\]
\end{teo}

Our next results are about the local times (sometimes called occupation times) of the random walk. For random interlacements, heuristically, the local time in~$x$ is the total number of visits to~$x$ of all the particles. Some previous results regarding local times of random interlacements, such as a Ray-Knight-type theorems and large deviations, can be found in \cite{local1} and \cite{local2}. 

\begin{teo}
\label{teotcl}
Let $\ell(x)$ be the local time of the one-dimensional random interlacements for $\in \zp$, then as $x\to \infty$
\[
\frac{\ell(x)-\alpha x^2}{x\sqrt{\alpha(4x-1)}}\lawlimit Z \sim \mathcal{N}(0,1 )
\]
\end{teo}

Finally, our last result is about the local times convergence for the random walk on the ring graph.   Our approach in the demonstration is constructing a coupling of local times with independent trajectories of the conditional random walk in $\zp$.

\begin{teo}
\label{teolocal}
Let $\ell(x)$ be the local time in~$x$ of the one-dimensional random interlacements for $\in \zp$, and $L_n(x)$ the local time in~$x$
 for the random walk in $R_{2n}$ up to time $\lfloor{4\alpha n^3}/{\pi^2}\rfloor$ conditioned on not hitting the origin.
 Then as $n\to \infty$
\[
L_n(x) \lawlimit \ell(x).
\]
\end{teo}
Observe that Theorem~\ref{teolimit} is also a corollary of Theorem~\ref{teolocal}. We opted to state the former one separately
because the proof of Theorem~\ref{teolimit} is much more
simple and straightforward than that of Theorem~\ref{teolocal}.

Plan of the paper: In section~2 we define capacity in one dimension and the conditional random walk, proving its martingale property. Then in section~2.3 we use these definitions to define the random interlacements process by means of the construction developed in \cite{teixeira} for random interlacements on weighted transient graphs. In section~2.4 we  define the local times of the process and calculate some important quantities related to them. In section~2.5 we  define the conditional walk on the ring graph (i.e. the torus in one dimension) and in section~2.6 we prove some auxiliary results  related to its law. In section~3 we  prove the three main theorems stated in the introduction.

\section{Preliminaries}
Unless stated otherwise, $(S_t)_{t \in \zp}$ will stand for the  simple random walk in one (or $d$) dimension(s) and $\pr_x$ will stand for the probability measure of the random walk starting in $x$. The inner boundary of a set $A\subset\zd$ is defined by $\partial A=\{x\in A\mid \|x-y\|_1=1 \mbox{ for some } y\notin A\}$ (as usual, $\|\cdot\|_p$
stands for the $L_p$-norm in $\mathbb{R}^d$). 
We also define the stopping times
  \[
 \tau_{A}=\inf\{t\geq 0\mid S_t \in A\}\quad\mbox{and}\quad\overline{\tau}_{A}=\inf\{t\geq 1\mid S_t \in A\}.
 \]
Abbreviate $\tau_{\{x\}}=\tau_x$ and $\overline{\tau}_{\{x\}}=\overline{\tau}_x$.  Define also the harmonic measure of $A$ as the entrance measure from infinity, so for $d\geq2$ it holds
 \[
 \hm_A(x)=\lim \limits_{|y|\to \infty}\pr_y[S_{\tau_A}=x].
 \]
 In dimension $d=1$ we have only two possible directions and then the entrance point have probabilities $1/2$ and $1/2$ on the two extremes of $A$.
 Finally, we define the diameter of a set by
 \[
 \diam(A)=\sup\limits_{x,y \in A}\|x-y\|_2
 \]
\subsection{Capacity in one dimension}
 To work with capacity in lower dimensions we need to make use of the potential kernel of the random walk. The potential kernel $a(x)$ for any random walk $X_t$ in $\zd$ is defined in section $4.4$ of \cite{rwb} by
  \[
  a(x)=\sum \limits_{k=0}^\infty \big(\pr[X_k=0]-\pr[X_k=x]\big).
  \]
 If the random walk is transient, then we can relate $a(x)$ to the Green function $G(x)$ by $a(x)=G(0)-G(x)$, but if the random walk is recurrent, the Green function does not exists. Theorem~$4.4.8$ of \cite{rwb} states that for simple random walk in one dimension the potential kernel is given by $a(x)=|x|$.

Now we define the capacity for one dimension, this definition is  analogous to the one of section~6.6 of~$\cite{rwb}$ for two dimensions and also used in~\cite{2dint}. The capacity of a set $A\subset \mathbb{Z}$ containing the origin can be defined as
\begin{equation}
\label {capdef}\cp(A):=\sum \limits_{x\in A} \hm_A(x)a(x-z)\mbox{, for any $z \in A$,}
\end{equation}
and for any other subset $B$ the capacity is given by the capacity of any translation of $B$ that contains the origin.

 As the harmonic measure can only be non null on the extremal points of a set, we have that $\cp(A)=\cp([\min A; \max A])$. Let us consider a interval $I=[-x,y]$ with both $x$,$y$ positive. Then by our definition of capacity using $z=0$:
 \[
\cp(A):= \hm_A(x)a(x)+\hm_A(y)a(y)=\frac{y+x}{2}.
\]
Here we used that if we start the walk from the positive side of $\mathbb{Z}$, then clearly it will enter~$I$ at~$y$ and the same can be said about the negative side and~$x$. This result and the translational invariance of    capacity imply that for any
finite subset~$A$ of~$\mathbb{Z}$
\[
\cp(A)=\frac{\diam(A)}{2}.
\]

\subsection{The conditional random walk}

Here we will construct random walks conditioned on never hitting $0$; since such a  walk never changes its sign, we can consider it only on $\zp$.
Let $x\geq 1$ be a positive integer. We start at $x$ a simple random walk $(X_t)_{t\geq 0}$ conditioned on not hitting $0$. To define the law of $X_t$ let us condition it on hitting $N$ before $0$ and take a limit in law. So for $x<N$ we have the following transition probabilities
\begin{align}
\label {crw}p_{x,x+1}&=1-p_{x,x-1}\\&=\pr_x[X_1=x+1|\tau_0 >\tau_N]\\
&=\frac{\pr_x[X_1=x+1]\pr_{x+1}[\tau_0 >\tau_N]}{\pr_{x+1}[\tau_0 >\tau_N]}\\&=\frac{x+1}{2x}.
\end{align}

A limit in $N$ takes off the restriction on $x$, and then  we obtain a law for any integer $x\geq 1$. It is interesting to compare this random walk with a simple random walk where we apply a Doob's $h$-transform using the potential kernel for one dimension $a(x)=|x|$.

The Doob's $h$-transform is applied using a function $h$ defined on the state space $S$ and a Markov chain with transitions $P(x,y)$, then the new transition probabilities $P^*(x,y)$ will be defined for all sites $x$ with $h(x)\neq0$ by
\[
P^*(x,y)=\frac{P(x,y)h(y)}{h(x)}.
\]

So when we apply the transform for the simple random walk $S_t$ with constant transition probabilities $P(x,x+1)=1/2$  we get a new random walk $X_t$ with transition probabilities given by
\[P^*(x,x+1)=\frac{a(x+1)P(x,x+1)}{a(x)}=\frac{x+1}{2x}
\]
The same walk could also be seen as a random walk with conductances on $\zp$, where the conductances are given by $c_{x,x+1}=a(x+1)=x+1$. 

 From now on $\cpr$ will stand for the probability 
measure of the conditional walk on~$\zp$. 

The following result helps a lot when doing calculations about the conditional random walk.

\begin{lemma}\label{martingale} Consider $(X_t)_{t\geq 0}$ a Markov chain in a space $S$ with transition probabilities $P(x,y)$.
 Let $(X^*_t)_{t\geq 0}$  be the Doob's h-transform of $(X_t)_{t\geq 0}$ with respect to a function $h$, which is nonnegative and harmonic outside the set $S^*=\{x \in S: \exists y\mbox{ with } P(x,y)>0\mbox{ and }h(y)=0\}$. Then the process  $((h(X^*_{t\wedge \tau_{S^*}}))^{-1})_{t\geq 0}$ with $h(X_0)\neq 0$ is a martingale.
\end{lemma}

\begin{proof}
Let us denote for all $x,y$ with $h(x)h(y)\neq 0$ the transition probabilities of $(X^*_t)_{t\geq 0}$ by $P^*(x,y)$.
Just a straightforward calculation here, being $\{\mathcal{F}_t\}_{t\geq0}$ the associated filtration for the process $((h(X_{t\wedge \tau_{S^*}}))^{-1})_{t\geq 0}$ we have:
\begin{align*}
\ex\left[(h(X^*_{(t+1)\wedge \tau_{S^*}}))^{-1}\vert \mathcal{F}_t\right]&=\sum \limits_{y\in S: h(y)\neq0} (h(y))^{-1}P^*(X^*_{t\wedge \tau_{S^*}},y)\\
&=\sum \limits_{y\in S: h(y)\neq0}\frac{1}{h(y)}\frac{P(X^*_{t\wedge \tau_{S^*}},y)h(y)}{h(X^*_{t\wedge \tau_{S^*}})}\\
&=\sum \limits_{y\in S: h(y)\neq0}\frac{P(X^*_{t\wedge \tau_{S^*}},y)}{h(X^*_{t\wedge \tau_{S^*}})}\\
&=(h(X^*_{t\wedge \tau_{S^*}}))^{-1}\pr[h(X^*_{(t+1)\wedge \tau_{S^*}})\neq0\mid \mathcal{F}_t]\\
&=(h(X^*_{t\wedge \tau_{S^*}}))^{-1}.
\end{align*}
Which complete the proof that the process is a martingale.
\end{proof}

For the conditional random walk $X_t$, we can use Lemma~\ref{martingale} to show that the process $\{(X_{t\wedge \tau_1})^{-1}\}_{t\geq 0}$ is a martingale. This fact will help us with calculations.

Next, we need
\begin{lemma}Let $X_t$ be the conditional random walk on $\zp$ started at $y$ with $N>y>x>1$. Then
\label{lemmacrw}
\begin{itemize}
\item[(i)] $\cpr_y[\tau_x<\tau_N]=\frac{x(N-y)}{y(N-x)}$,
\item[(ii)] $\cpr_y[\tau_x< \infty]=\frac{x}{y}$,
\item[(iii)] $\cpr_x[\overline{\tau}_x= \infty]=\frac{1}{2x}$.
\end{itemize}
\end{lemma}
\begin{proof}
These are also very straightforward calculations using the optional stopping theorem. The first result comes from using the martingale $(X_{t\wedge \tau_1})^{-1}$ with the stopping time $\tau(N)\wedge \tau(x)$:
\[
\frac{1}{y}=\frac{1}{x}\cpr_y[\tau_x<\tau_N]+\frac{1}{N}(1-\cpr_y[\tau_x<\tau_N]),
\]
and then isolating the probability in the expression give us the desired result. For the second relation we just need to take limit in $N$, as clearly $\tau_N$ will diverge. For 
the last expression we observe that the first step should be to~$x+1$ and then
\begin{align*}
\cpr_x[\overline{\tau}_x= \infty]&=P(x,x+1)\cpr_{x+1}[\overline{\tau}_x= \infty]\\
&=\frac{x+1}{2x}\Big(1-\frac{x}{x+1}\Big)=\frac{1}{2x}.
\end{align*}
This concludes the proof.
\end{proof}

An important consequence of Lemma~\ref{lemmacrw}~(iii) is that the random walk conditioned on never hitting the origin is transient.

\subsection{Construction of the process}

To define the process, we will use the construction of~\cite{teixeira}, where the process of random interlacements is constructed for any weighted transient graph (i.e., a graph on which the random walk is transient). The graph considered here is $\mathbb{Z}$, so our weights (or conductances) $a_{x,y}$ are only positive if $|x-y|=1$. The conductances that generate the conditional random walk defined in $\eqref{crw}$ are
\[
a_{x,x+1}=a_{x+1,x}=\frac{x(x+1)}{2}.
\]
Then, the random walk on the graph with conductances is reversible with reversible measure $\mu_x:=a_{x,x+1}+a_{x,x-1}=x^2$, and its transition probabilities are 
\[
P(x,x+1)=\frac{a_{x,x+1}}{\mu_x}=\frac{x+1}{2x},
\]
as it should be. In accordance to~\cite{teixeira}, the capacity
 (denoted here by $\cpt(A)$) with respect to the conditional walk is defined by
\begin{equation}
\label{captex} \cpt(A)=\sum \limits_{x \in A} e_A(x),
\end{equation}
where $e_A(x)$ is the \textit{equilibrium measure} defined by
\[
e_A(x)=\ind[x\in A]\pr_x[\overline{\tau}_A=\infty]\mu_x.
\]
This definition uses the equilibrium measure of the conditional random walk, so its straightforward to see that for any finite $A\subset \mathbb{Z}$ we have $\cpt(A)=\cpt(A\cup \{0\})$.

We now show that for any set $A$ containing the origin we have
\[
\cpt(A)=\cp(A).
\]
Since the capacity of any finite set~$A$ is the same as the capacity of the shortest interval containing it, we can assume without loss of generality that $A=[a,b]$. In the definition of $\cp(A)$ we consider a set containing the origin, and for any other set we consider a translation of it containing the origin (this capacity is translation invariant). So we consider here $a<0$ and $b>0$, and then by Lemma~\ref{lemmacrw} we have
\begin{align}
\nonumber \cpt([a,b])&=\sum \limits_{x=a}^b e_{[a,b]}(x)\\
\nonumber&=e_{[a,b]}(a)+e_{[a,b]}(b)\\
\nonumber&=\pr_a[\overline{\tau}_{a}=\infty]a^2+\pr_b[\overline{\tau}_{b}=\infty]b^2\\
\nonumber&=\frac{b-a}{2}=\cp([a,b]).
\end{align}
With this we can relate both capacities by
\[
\cpt(A)=\cp(A\cup\{0\}).
\]
Let $W^*$ be the space of doubly infinite trajectories that spend a finite time in each finite set, the random interlacements process is defined as a Poison point process on the space $W^*\times\mathbb{R}^+$ with intensity by a measure $\nu \cdot \lambda$, where $\lambda$ is the Lebesgue measure on $\mathbb{R}^+$ and $\nu$ is a measure on $W^*$ characterized by
\[
\nu(\{\omega^* \in W^* :\omega(\mathbb{Z})\cap A \neq \O\})=\cpt(A).
\]
With this we can characterize the law of the process as 
\begin{align*}
\pr[A\subset \vac]&=\exp\{-\nu(\{\omega^* \in W^* :\omega(\mathbb{Z})\cap A \neq \O\})\times\lambda([0,\alpha])\}\\
&=\exp\{-\alpha \cpt(A)\}\\
\end{align*}
For a complete description of the construction see \cite{teixeira}.

One property of the construction that will be useful is that the number of trajectories that hit a set $A$ (we will denote it by $N_A$) in the random interlacements process at level $\alpha$ is such that
\[
N_A \sim {\rm Poisson}(\alpha \cp(A\cup\{0\})).
\]

\subsection{Local times}

The local time or occupation time of a transient random walk in site $x$ can be defined as the time the random walk spends at site $x$. The local time of the random interlacements process will be the sum of the local times of each trajectory. We know we have $N_{\{x\}}$ trajectories that hit $x$. For each of those trajectories, from a visit to $x$ a particle has a constant probability of escaping, making the local time of each particle a geometric random variable with success probability $(2x)^{-1}$ (see Lemma~\ref{lemmacrw}). So the random interlacements local time of $x$ is a compound Poisson variable 
\[
\ell(x)=\sum \limits_{k=1}^{N_{\{x\}}}V_k, 
\]
where $V_k$ are i.i.d. ${\rm Geometric((2x)^{-1})}$ with $\pr[V_k=j]=\frac{1}{2x}\big(1-\frac{1}{2x}\big)^{j-1}$.

\begin{lemma}
\label{lemmacharac}
 For any $x>0$ the characteristic function of $\ell(x)$ is
 \[
   \phi_{\ell(x)}(t)=\exp\left\{{\alpha x^2}\frac{(e^{it}-1)}{2x-(2x-1)e^{it}}\right\}.
 \]

\end{lemma}
\begin{proof}Straightforward calculation of the characteristic function of a compound Poisson of geometrics with $N_{\{x\}}$ being a Poisson variable with parameter $\alpha\cp(\{0,x\})=\alpha x/2$. The characteristic function of geometric variables is also well-known, so we have
 \begin{align}
\nonumber\phi_{\ell(x)}(t)&=\exp\left\{\frac{\alpha x}{2}\left(\frac{\frac{1}{2x}e^{it}}{1-\Big(1-\frac{1}{2x}\Big)e^{it}}-1\right)\right\}\\
\label{charac1} &=\exp\left\{{\alpha x^2}\frac{(e^{it}-1)}{2x-(2x-1)e^{it}}\right\}.
 \end{align}
 This characteristic function then defines the law of the local time $\ell(x)$.
\end{proof}

\subsection{Random walk on the ring graph}

In higher dimensions it is known that we can approximate the trace left by the random walk in a torus by the random interlacements process, main results about this can be found in \cite{torusint} and more recently for two-dimensional random interlacements in \cite{2dint}. Here we wish to establish the same relation in dimension one.

The one-dimensional torus is simply the ring graph, we will denote it by $R_n=(S_n,E_n)$, where $S_n=\{0,1,\ldots,n-1\}$ and $E_n=\{({0},{1});({1},{2});\ldots,({n-1},{0})\}$. Here we consider also a simple random walk conditioned on not hitting the origin until time $t$, we denote its law by $\rpr^t$ and its respective vacant set by $V_t$. This random walk can be seen as a random walk on~$\mathbb{Z}$ conditioned to not hitting $0$ and $n$. Let us define the quantity
\[
h_n(x,t)=\pr_x[\tau_{\{0,n\}}>t].
\]
Then the law for the conditioned walk that runs until time $s$ is given by
\begin{align*}
\rpr^s_x[X_1=x+1]&=1-\rpr^s_x[X_1=x-1]=\pr_x[X_1=x+1|\tau_{\{0,n\}}>s]\\
&=\frac{\pr_x[X_1=x+1,\tau_{\{0,n\}}>s]}{\pr_x[\tau_{\{0,n\}}>s]}\\
&=\frac{h_n(x+1,s-1)}{2h_n(x,s)},
\end{align*}
and then the probability of a path $\gamma$ of size $m$ starting at $x$ when the remaining time is $t$ is given by
\begin{equation}
\label{lawring}\rpr^t_{x}[\gamma]=\frac{h_n(\gamma_{m},t-m)}{2^{m}h_n(x,t)}.
\end{equation}
Observe also that the conditional random walk law for the same path (as it is always a valid path in $\zp$) is
\begin{equation}
\label{lawcond}\cpr_{x}[\gamma]=\frac{\gamma_{m}}{2^{m}x}.
\end{equation}

Now we need to understand better the asymptotic behavior of $h_n(x,t)$, this will be crucial in all the results we will prove.

\subsection{On the probability of not hitting~$0$ on the ring}

We now analyze $h_n(x,t)$. First we present an application of result 5.7 from chapter~XIV of~\cite{feller}.
\begin{lemma} Consider the simple random walk in $\mathbb{Z}$, if $0<x<n$ it holds that 
\label{lemmaruin}
\[
\pr_x[\tau_0<\tau_n,\,\tau_0=k ]=\frac{1}{n}\sum \limits_{j=1}^{n-1} \cos^{k-1}\Big(\frac{\pi j}{n}\Big) \sin\Big(\frac{\pi j}{n}\Big)\sin\Big(\frac{\pi x j}{n}\Big).
\]
\end{lemma}
 The utility of these results comes from the fact that in includes the symmetric case $\pr_x[\tau_0>\tau_n,\,\tau_n=k ]=\pr_{n-x}[\tau_0<\tau_n,\,\tau_0=k ]$ and so we can write
\begin{equation}
\label{hr}\pr_x[\tau_{\{0,n\}}=k]=\pr_x[\tau_0<\tau_n,\,\tau_0=k ]+\pr_x[\tau_0>\tau_n,\,\tau_n=k ].
\end{equation}

With this we can obtain an expression for $h_n(x,t)$:
\begin{lemma} For any integer $x \in [0,n]$ it holds
\label{lemmah}
\[
h_n(x,t)=\frac{2}{n}\sum \limits_{j=1}^{\lfloor {n}/{2}\rfloor}\cos^{t}\Big(\frac{\pi (2j-1)}{n}\Big) \cot\Big(\frac{\pi (2j-1)}{2n}\Big)\sin\Big(\frac{\pi x (2j-1)}{n}\Big).
\]
\end{lemma}
\begin{proof}
Observe that
\begin{align*}
\pr_x[\tau_0>\tau_n,\,\tau_n=k ]&=\pr_{n-x}[\tau_0<\tau_n,\,\tau_0=k ]\\
&=\frac{1}{n}\sum \limits_{j=1}^{n-1} \cos^{k-1}\Big(\frac{\pi j}{n}\Big) \sin\Big(\frac{\pi j}{n}\Big)\sin\Big(\frac{\pi (n-x) j}{n}\Big)\\
&=\frac{1}{n}\sum \limits_{j=1}^{n-1} (-1)^{j+1} \cos^{k-1}\Big(\frac{\pi j}{n}\Big) \sin\Big(\frac{\pi j}{n}\Big)\sin\Big(\frac{\pi x j}{n}\Big).
\end{align*}
So when we sum the probabilities in $\eqref{hr}$ and use Lemma~\ref{lemmaruin}, all terms with even $j$'s disappear and we get 
\begin{align*}
\pr_x[\tau_{\{0,n\}}=k]&=\pr_x[\tau_0<\tau_n,\,\tau_0=k ]+\pr_x[\tau_0>\tau_n,\,\tau_n=k ]\\
&=\frac{2}{n}\sum \limits_{j=1}^{\lfloor {n}/{2}\rfloor} \cos^{k-1}\Big(\frac{\pi (2j-1)}{n}\Big) \sin\Big(\frac{\pi (2j-1)}{n}\Big)\sin\Big(\frac{\pi x (2j-1)}{n}\Big),
\end{align*}
and then finally
\begin{align*}
h_n(x,t)&=\frac{2}{n}\sum \limits_{k=t+1}^\infty\sum \limits_{j=1}^{\lfloor {n}/{2}\rfloor} \cos^{k-1}\Big(\frac{\pi (2j-1)}{n}\Big) \sin\Big(\frac{\pi (2j-1)}{n}\Big)\sin\Big(\frac{\pi x (2j-1)}{n}\Big)\\
&= \frac{2}{n}\sum \limits_{j=1}^{\lfloor {n}/{2}\rfloor}\frac{\sin\Big(\frac{\pi (2j-1)}{n}\Big)}{1-\cos\Big(\frac{\pi (2j-1)}{n}\Big)}\sin\Big(\frac{\pi x (2j-1)}{n}\Big)  \cos^{t}\Big(\frac{\pi (2j-1)}{n}\Big)\\
&= \frac{2}{n}\sum \limits_{j=1}^{\lfloor {n}/{2}\rfloor}\cos^{t}\Big(\frac{\pi (2j-1)}{n}\Big) \cot\Big(\frac{\pi (2j-1)}{2n}\Big)\sin\Big(\frac{\pi x (2j-1)}{n}\Big).
\end{align*}

This concludes the proof.
\end{proof}

Our main concern now is to turn the expression in Lemma~\ref{lemmah} into something tractable. In order to do this we first need to demonstrate that the only term asymptotically relevant in the sum is the first one.

\begin{teo} If $t=t(n)$ satisfies
\begin{equation}
\label{cond}\liminf \limits_{n\to \infty}\frac{t}{n^2 \ln n}\geq \frac{4}{\pi^2},
\end{equation}
then the asymptotic behavior of $h_n(x,t)$ as $n\to \infty$ is
\label{asymp}
\begin{align*}
h_n(x,t)&=\Big(1+\ord(n^{-2})\Big)\frac{4}{\pi}\cos^{t}\Big(\frac{\pi }{n}\Big) \sin\Big(\frac{\pi x}{n}\Big)\\
&\sim \frac{4}{\pi}\cos^{t}\Big(\frac{\pi }{n}\Big) \sin\Big(\frac{\pi x}{n}\Big).
\end{align*}

\end{teo}
\begin{proof}
First we get an upper bound for the sum without the first term. Let us denote the first term by $T_1$; using the fact that $\cos$ and $\cot$ are decreasing functions on $[0,\pi/2]$ we have
\begin{align*}
\mid h_n(x,t)-T_1\mid&=\left\vert\frac{2}{n}\sum \limits_{j=2}^{\lfloor \frac{n}{2}\rfloor}\cos^{t}\Big(\frac{\pi (2j-1)}{n}\Big) \cot\Big(\frac{\pi (2j-1)}{2n}\Big)\sin\Big(\frac{\pi x (2j-1)}{n}\Big)\right\vert \\
&\leq \frac{2}{n}\sum \limits_{j=2}^{\lfloor \frac{n}{2}\rfloor}\left\vert\cos^{t}\Big(\frac{\pi (2j-1)}{n}\Big)\right\vert \left\vert\cot\Big(\frac{\pi (2j-1)}{2n}\Big)\right\vert \\
&\leq \frac{2}{n}\Big(\left\lfloor \frac{n}{2}\right\rfloor-1\Big) \left\vert\cos^{t}\Big(\frac{2\pi}{n}\Big)\right\vert \left\vert\cot\Big(\frac{\pi }{n}\Big)\right\vert \\
&\leq \left\vert\cos^{t}\Big(\frac{2\pi}{n}\Big)\right\vert \left\vert\cot\Big(\frac{\pi }{n}\Big)\right\vert.
\end{align*}
Dividing both sides by $T_1$ we get
\begin{align}
\label{bound3} \left\vert\frac{h_n(x,t)}{T_1}-1\right\vert&\leq \left\vert\frac{\cos^{t}\Big(\frac{2\pi}{n}\Big)}{\cos^{t}\Big(\frac{\pi}{n}\Big)}\right\vert \left\vert\frac{\cot\Big(\frac{\pi }{n}\Big)}{\cot\Big(\frac{\pi }{2n}\Big)}\right\vert\frac{1}{\left\vert \sin\Big(\frac{\pi x}{n}\Big)\right\vert}.
\end{align}
Let us study the asymptotic behavior of the right-hand side of $\eqref{bound3}$. For cosine we have when $x\to 0$ that $\cos(x)=e^{-x^2/2}(1+\ord(x^4))$ and for cotangent $\cot(x)=x^{-1}(1+\ord(x^2))$. Using this in the fractions we get
\begin{align*}
\frac{\cos^{t}\Big(\frac{2\pi}{n}\Big)}{\cos^{t}\Big(\frac{\pi}{n}\Big)}&=(1+\ord(tn^{-4}))e^{-{3t\pi^2}/{(2n^2)}}\quad\mbox{and}\quad\frac{\cot\Big(\frac{\pi }{n}\Big)}{\cot\Big(\frac{\pi }{2n}\Big)}=\frac{1}{2}(1+\ord(n^{-2})).
\end{align*}
With this
\begin{align*}
\left\vert\frac{h_n(x,t)}{T_1}-1\right\vert&\leq \left\vert\frac{\cos^{t}\Big(\frac{2\pi}{n}\Big)}{\cos^{t}\Big(\frac{\pi}{n}\Big)}\right\vert \left\vert\frac{\cot\Big(\frac{\pi }{n}\Big)}{\cot\Big(\frac{\pi }{2n}\Big)}\right\vert\frac{1}{\left\vert \sin\Big(\frac{\pi x}{n}\Big)\right\vert}\\
&=(1+\ord(tn^{-4})+\ord(x^2n^{-2}))\frac{1}{2\sin\Big(\frac{\pi x}{n}\Big)}e^{-{3t\pi^2}/{(2n^2)}}.
\end{align*}
If $x\sim cn$, then our bound is $\ord(e^{-{3t\pi^2}/{(2n^2)}})$ and if $x=o(n)$ then our bound is $\ord(nx^{-1}e^{-{3t\pi^2}/{(2n^2)}})$. As we do not need such refined results here, we will just work with the worst case bound, so
\begin{align*}
\left\vert\frac{h_n(x,t)}{T_1}-1\right\vert&= \ord(ne^{-{3t\pi^2}/{(2n^2)}}).
\end{align*}
Since for sufficiently large $n$ we have $\frac{3t\pi^2}{2n^2}\geq 3\ln n$, then ${n}e^{-{3t\pi^2}/{(2n^2)}}\leq n^{-2}$. Therefore
\begin{align}
\label{exact} h_n(x,t)&=(1+\ord(n e^{-{3t\pi^2}/{(2n^2)}}))T_1\\
\nonumber&=(1+\ord(n^{-2}))\frac{2}{n}\cos^{t}\Big(\frac{\pi}{n}\Big) \cot\Big(\frac{\pi}{2n}\Big)\sin\Big(\frac{\pi x}{n}\Big),
\end{align}
and again using the asymptotic expression of $\cot x$ we have 
\begin{align*}
\frac{2}{n}\cot\Big(\frac{\pi}{2n}\Big)=\frac{2}{n}\Big(\frac{2n}{\pi}+\ord (n^{-1})\Big)= \frac{4}{\pi}(1+\ord(n^{-2})).
\end{align*}
and this gives us the asymptotics for $h_n(x,t)$.
\end{proof}

This concludes the preliminaries we need in order to prove our main results.

\section{Proofs of the main results}
\subsection{Proof of Theorem~\ref{teolimit}}
We want to show that the random walk on the torus conditioned on not hitting the origin for a fixed time has as a limit the random interlacements process when we look at a fixed subset around the origin. The main question here is about how much time the conditional walk on the ring graph (the one-dimensional discrete torus) needs in order to match the random interlacements behavior. It turns out that here the time for the random interlacements convergence will be
\[
t_{\alpha,n}=\frac{\alpha n^3}{2\pi^2},
\]

Now let us begin the proof.
\begin{proof}[Proof of Theorem~\ref{teolimit}]
Although we already said which is the right value for $t_{n,\alpha}$, here we will work with a generic $t$ and then find the right value. The only assumption we have here is that $t$ should satisfy the condition of Theorem~\ref{asymp}.

Our aim here is to find a $t$ such that for any fixed interval $[-a,b]$ with $a,b>0$ we have for a starting point $x=\lfloor n/2\rfloor$ outside the interval
\begin{equation*}
  \rpr^{t}_x\big[[-a,b]\subset V_{t}\big] \to \exp\left\{-\frac{\alpha(a+b)}{2}\right\} .
\end{equation*}
The choice of $x=\lfloor n/2\rfloor$ is to keep the walk starting sufficiently away from the limit points of the interval so that the initial points do not affect the law of the vacant set. 
For this consider the conditional random walk on the ring graph as a random walk on $\mathbb{Z}$ conditioned on not hitting $0$ or $n$ for time $t$, so the site $-a$ will be equivalent of point $n-a$ in this walk.

With this we have
\begin{align}
  \nonumber \rpr^{t}_x\big[[-a,b]\subset V_{t}\big]&=\pr_x[\tau_{\{b,n-a\}}>t\mid \tau_{\{0,n\}}>t]\\
  &\label{fract1}=\frac{\pr_x[\tau_{\{b,n-a\}}>t]}{\pr_x[\tau_{\{0,n\}}>t]};
\end{align}
here we used that the simple random walk necessarily needs to hit $\{b,n-a\}$ in order to hit $\{0,n\}$. Now that we are working with the simple random walk, we can use its symmetry to rewrite the probability in the numerator of $\eqref{fract1}$:
\[
\pr_x[\tau_{\{b,n-a\}}>t]=\pr_{x-b}[\tau_{\{0,n-a-b\}}>t].
\]
Using this in $\eqref{fract1}$, we obtain
\begin{align*}
\rpr^{t}_x\big[[-a,b]\subset V_{t}\big]&=\frac{h_{n-a-b}(x-b,t)}{h_n(x,t)}.
\end{align*}
Now using Theorem~\ref{asymp} we have
\begin{align*}
  \nonumber \rpr^{t}_x\big[[-a,b]\subset V_{t}\big]&\sim\frac{\cos^{t}\Big(\frac{\pi }{n-a-b}\Big) \sin\Big(\frac{\pi (x-b)}{n-a-b}\Big)}{\cos^{t}\Big(\frac{\pi }{n}\Big) \sin\Big(\frac{\pi x}{n}\Big)}.
\end{align*}

As $x=\lfloor n/2\rfloor$, both sines in the above expression are asymptotic to $1$. Then, using again the asymptotic relation $\cos x \sim e^{-x^2/2}$ as $x\to 0$
\begin{align*}
  \nonumber \rpr^{t}_x\big[[-a,b]\subset V_{t}\big]&\sim\frac{\exp\Big\{\frac{-t\pi^2 }{2(n-a-b)^2}\Big\} }{\exp\Big\{\frac{-t\pi^2 }{2n^2}\Big\} }\\
    &\sim\exp\Big\{-\frac{t\pi^2(a+b)}{n^3}\Big\}.
\end{align*}

so  by continuity of the exponential function we need this exponent to be asymptotic to $\alpha(a+b)/2$, this means
\begin{equation}
\label{time}t\sim \frac{\alpha n^3}{a\pi^2}.
\end{equation}
 To finally conclude that everything we did was valid, we just need to see that this value of $t$ satisfies the condition on Theorem~\ref{asymp}, then for the value of $t$ in $\eqref{time}$ 
\begin{align*}
\liminf \limits_{n\to \infty}\frac{t}{n^2 \ln n}=\infty.
\end{align*}
This concludes the proof of Theorem~\ref{teolimit}, as for any fixed interval $[-a,b]$

\begin{equation*}
  \rpr^{{\alpha n^3}/{(a\pi^2)}}_z\big[[-a,b]\subset V_{\lfloor{\alpha n^3}/{(a\pi^2)}\rfloor}\big] \sim \exp\left\{-\frac{\alpha(a+b)}{2}\right\}.
\end{equation*}
Therefore,
\[
V_{\lfloor{\alpha n^3}/{(a\pi^2)}\rfloor}\lawlimit \vac,
\]
as desired.
\end{proof}
\subsection{Proof of Theorem~\ref{teotcl}}

Now we prove the central limit theorem for the local times. This will be done using the characteristic function of the local time together with the L\'evy's continuity theorem.

\begin{proof}[Proof of Theorem~\ref{teotcl}]
For the characteristic function of the local time, using Lemma~\ref{lemmacharac} we have
\[
  \phi_{\ell(x)}(t)=\exp\left\{{\alpha x^2}\frac{(e^{it}-1)}{2x-(2x-1)e^{it}}\right\}.
  \]
We need to study the asymptotic behavior of the exponent when $t\to 0$ and $x\to \infty$. Here we do not only need the main term, but also the error to see under which conditions the convergence holds. We first centralize our variable
 \begin{align}
  \nonumber\phi_{\ell(x)-\alpha x^2}(t)&=\exp\left\{{\alpha x^2}\frac{(e^{it}-1)}{2x-(2x-1)e^{it}}-\alpha x^2 it\right\}\\
\label{characzero}  &=\exp\left\{{\alpha x^2}\left(\frac{(e^{it}-1)}{2x-(2x-1)e^{it}}- it\right)\right\}.
  \end{align}
  Now we manipulate the term inside the brackets of $\eqref{characzero}$
   \begin{align}
\label{brack}\frac{(e^{it}-1)}{2x-(2x-1)e^{it}}- it&=\frac{(e^{it}-1)-it(2x-(2x-1)e^{it})}{2x-(2x-1)e^{it}},
  \end{align}
 and use the asymptotic expansion
\begin{align*}
e^{it}-1=it-\frac{t^2}{2}+\ord(t^3),
\end{align*}
and also
\begin{align*}
{2x-(2x-1)e^{it}}&=1+(2x-1)(1-e^{it})\\
&=1-\ord(xt).
\end{align*}
Then $\eqref{brack}$ becomes:
\begin{align*}
\frac{(e^{it}-1)-it(2x-(2x-1)e^{it})}{2x-(2x-1)e^{it}}&=\frac{(it-\frac{t^2}{2}+\ord(t^3))-it(2x-(2x-1)(1+it-\frac{t^2}{2}+\ord(t^3)))}{1-\ord(xt)}\\
&=-\frac{t^2}{2}(4x-1)+\ord(x^2t^3),
\end{align*}
and then, coming back to $\eqref{characzero}$,
 \begin{align}
  \nonumber\phi_{\ell(x)-\alpha x^2}(t)&=\exp\left\{{\alpha x^2}\left(-\frac{t^2}{2}(4x-1)+\ord(x^2t^3)\right)\right\}\\
  &=\exp\left\{-\frac{\alpha (4x-1) x^2t^2}{2}+\ord(x^4t^3)\right\}.
  \end{align}
Finally dividing the variable by the standard deviation  $\sqrt{\alpha(4x-1)}x$:
\begin{align*}
\phi_{\ell^*(x)}(t)&=\exp\left\{-\frac{t^2}{2}+\ord\left(\frac{t^3}{\sqrt{x}}\right)\right\},\mbox{ where }\ell^*(x)=\frac{\ell(x)-\alpha x^2}{\sqrt{\alpha(4x-1)}x}.
  \end{align*}
Then as $x\to \infty$ this characteristic function converges to the standard normal distribution one, and by the continuity theorem (see Theorem~$9.5.2$ of \cite{ppath}), we conclude our proof.
\end{proof}
\subsection{Proof of Theorem~\ref{teolocal}}

In order to prove Theorem~\ref{teolocal} we need more preliminaries. Again we will consider the conditional random walk in the ring graph as the simple random walk in $\mathbb{Z}$ conditioned to the event $\tau_{0,n}>t^*$, where $t^*=\alpha n^3/(2\pi^2)$ is the random interlacements convergence time.

\begin{lemma}
\label{lemmamid}
Consider the conditional random walk on the ring graph with $2n$ sites. Let $t=t(n)$ and $\Delta=\Delta(n)$ be such that $t$, $\Delta$ and $t-\Delta$ satisfy condition $\eqref{cond}$. For any $1<x<n$ it holds that
\[
\rpr^t_x[\tau_n>\Delta]\leq (1+\ord(\Delta n^{-4}))\frac{8}{\pi}\cos\Big(\frac{\pi x}{2n}\Big)\exp\Big\{-\frac{3\pi^2\Delta}{8n^2}\Big\}.
\]
\end{lemma}

\begin{proof} Splitting the above probability into the sum of the probabilities of each path, we consider the set $\Gamma$ of paths which does not include the sites $0$ or $n$ and have length~$\Delta$ . Then
\begin{align*}
\rpr^t_x[\tau_n>\Delta]&=\sum \limits_{\gamma \in \Gamma} \rpr^t_x[\gamma]\\
&=\sum \limits_{\gamma \in \Gamma} \frac{h_{2n}(\gamma_\Delta ,t-\Delta)}{2^{\Delta}h_{2n}(x,t)}.
\end{align*}
As $t$ and $t-\Delta$ satisfy condition $\eqref{cond}$, we can use Theorem~\ref{asymp} to obtain that
\begin{align*}
\rpr^t_x[\tau_n>\Delta]&=(1+\ord(n^{-2}))\sum \limits_{\gamma \in \Gamma}\cos^{-\Delta}\Big(\frac{\pi }{2n}\Big) \frac{\sin\Big(\frac{\pi \gamma_\Delta}{2n}\Big)}{2^{\Delta}\sin\Big(\frac{\pi x}{2n}\Big)}\\
&=(1+\ord(n^{-2}))\frac{\cos^{-\Delta}\Big(\frac{\pi }{2n}\Big)}{\sin\Big(\frac{\pi x}{2n}\Big)}\ex_x\Big[ \sin\Big(\frac{\pi X_\Delta}{2n}\Big) \ind[\tau_{\{0,n\}}>\Delta]\Big]\\
&\leq(1+\ord(n^{-2}))\frac{\cos^{-\Delta}\Big(\frac{\pi }{2n}\Big)}{\sin\Big(\frac{\pi x}{2n}\Big)}h_n(x,\Delta).
\end{align*}
Now, as $\Delta$ also satisfies condition $\eqref{cond}$, we again use Theorem~\ref{asymp} and get 
\begin{align*}
\rpr^t_x[\tau_n>\Delta]&\leq(1+\ord(n^{-2}))\frac{4\sin\Big(\frac{\pi x}{n}\Big)}{\pi\sin\Big(\frac{\pi x}{2n}\Big)}\frac{\cos^{\Delta}\Big(\frac{\pi }{n}\Big)}{\cos^{\Delta}\Big(\frac{\pi }{2n}\Big)}\\
&=(1+\ord(n^{-2}))\frac{8}{\pi}\cos\Big(\frac{\pi x}{2n}\Big)\frac{\cos^{\Delta}\Big(\frac{\pi }{n}\Big)}{\cos^{\Delta}\Big(\frac{\pi }{2n}\Big)}.
\end{align*}
Using the asymptotic expansion of cosine $\cos(x)= e^{-x^2/2}(1+\ord(x^4))$ we have
\begin{align*}
\rpr^t_x[\tau_n>\Delta]&\leq(1+\ord(n^{-2}))(1+\ord(\Delta n^{-4}))\frac{8}{\pi}\cos\Big(\frac{\pi x}{2n}\Big)\exp\Big\{-\frac{\pi^2\Delta}{2}\Big(\frac{1}{n^2}-\frac{1}{4n^2}\Big)\Big\}\\
&=(1+\ord(\Delta n^{-4}))\frac{8}{\pi}\cos\Big(\frac{\pi x}{2n}\Big)\exp\Big\{-\frac{3\pi^2\Delta}{8n^2}\Big\}.
\end{align*}
This concludes the proof.
\end{proof}

\begin{lemma}
\label{lemmapi4}
  Let $\{X_t\}_{t\in \zp}$ be a simple random walk on $\mathbb{Z}$. Consider $\Delta=\Delta(n)$ satisfying condition $\eqref{cond}$. Then for any $a\in\{1,\ldots, n-1\}$ we have
\begin{equation*}
\ex_a\Big[\sin\Big(\frac{\pi X_\Delta}{n}\Big)\mid \tau_{\{0,n\}}>\Delta\Big]=(1+\ord(n^{-2})) \frac{\pi}{4}.
\end{equation*}
\end{lemma}

\begin{proof} 
Consider a quantity $t>\Delta$ such that $t-\Delta$ satisfies $\eqref{cond}$. Then we have 
\begin{align*}
h_n(a,t)&=\pr_a[\tau_{\{0,n\}}>t]\\
&=\pr_a[\tau_{\{0,n\}}>\Delta]\pr_a[\tau_{\{0,n\}}>t\mid\tau_{\{0,n\}}>\Delta].
\end{align*}
Using the Markov property we obtain
\begin{align*}
\pr_a[\tau_{\{0,n\}}>t\mid\tau_{\{0,n\}}>\Delta]&=\ex_a[\pr_{X_{\Delta}}[\tau_{\{0,n\}}>t-\Delta]\mid\tau_{\{0,n\}}>\Delta]\\
&=\ex_a[h_n(X_\Delta,t-\Delta)\mid\tau_{\{0,n\}}>\Delta].
\end{align*}
That gives the relation
\begin{align*}
h_n(a,t)=h_n(a,\Delta)\cdot\ex_a[h_n(X_\Delta,t-\Delta)\mid\tau_{\{0,n\}}>\Delta] \intertext{or, equivalently}
\ex_a[h_n(X_\Delta,t-\Delta)\mid\tau_{\{0,n\}}>\Delta]=\frac{h_n(a,t)}{h_n(a,\Delta)}.
\end{align*}
Here, as $t$, $\Delta$ and $t-\Delta$ satisfy $\eqref{cond}$, we can use Theorem~\ref{asymp} and obtain
\begin{align*}
\ex_a\Big[\frac{4}{\pi}\cos^{t-\Delta}\Big(\frac{\pi }{n}\Big) \sin\Big(\frac{\pi X_\Delta}{n}\Big)\mid\tau_{\{0,n\}}>\Delta\Big]=(1+\ord(n^{-2}))\frac{\cos^{t}\Big(\frac{\pi }{n}\Big) \sin\Big(\frac{\pi a}{n}\Big)}{\cos^{\Delta}\Big(\frac{\pi }{n}\Big) \sin\Big(\frac{\pi a}{n}\Big)}.
\end{align*}
Rearranging the terms, we obtain
\begin{align*}
\ex_a\Big[\sin\Big(\frac{\pi X_\Delta}{n}\Big)\mid\tau_{\{0,n\}}>\Delta\Big]=(1+\ord(n^{-2})) \frac{\pi}{4},
\end{align*} 
which concludes the proof of Lemma~\ref{lemmapi4}.
\end{proof}

\begin{lemma} Consider the conditional random walk on the ring graph with $2n$ sites, $\mathbb{Z}/2n\mathbb{Z}$. Assume that $x=x(n)$, $t=t(n)$ and $\Delta=\Delta(n)$ are such that $x=o(n)$ and both $\Delta$ and $t-\Delta$ satisfy $\eqref{cond}$. Then we have
\[
\rpr^t_n[\tau_x>\Delta]=(1+\ord(n^{-1})+\ord( n^{-4}\Delta))\exp\Big\{-\frac{\Delta x\pi^2}{8n^3}\Big\}.
\]
\label{lemmapiece}
\end{lemma}
\begin{proof}
First we calculate the probability that there will be no visits to a fixed site $x$ in this interval by the random walk with initial site $n$:
\begin{align*}
\rpr^t_n[\tau_x>\Delta]&=\pr_n[\tau_x>\Delta\mid\tau_{\{0,2n\}}>t]\\
&=\frac{\pr_n[\tau_{\{x,2n\}}>\Delta]\pr_n[\tau_{\{0,2n\}}>t\mid \tau_{\{x,2n\}}>\Delta]}{\pr_n[\tau_{\{0,2n\}}>t]}\\
&=\frac{\pr_{n-x}[\tau_{\{0,2n-x\}}>\Delta]\pr_n[\tau_{\{0,2n\}}>t\mid \tau_{\{x,2n\}}>\Delta]}{\pr_n[\tau_{\{0,2n\}}>t]}\\
&=\frac{h_{2n-x}(n-x,\Delta)}{h_{2n}(n,t)}\pr_n[\tau_{\{0,2n\}}>t\mid \tau_{\{x,2n\}}>\Delta]\\
&=\frac{h_{2n-x}(n-x,\Delta)}{h_{2n}(n,t)}\ex_n[h_{2n}(X_\Delta,t-\Delta)\mid \tau_{\{x,2n\}}>\Delta].
\end{align*}
We can use Theorem~\ref{asymp} to obtain
\begin{align*}
&\frac{h_{2n-x}(n-x,\Delta)}{h_{2n}(n,t)}\ex_n[h_{2n}(X_\Delta,t-\Delta)\mid \tau_{\{x,2n\}}>\Delta]\\
&=(1+\ord(n^{-2}))\frac{\frac{4}{\pi}\cos^{\Delta}\Big(\frac{\pi}{2n-x}\Big)\sin\Big(\frac{\pi(n-x)}{2n-x}\Big)}{\frac{4}{\pi}\cos^{t}\Big(\frac{\pi}{2n}\Big)}\ex_n\Big[\frac{4}{\pi}\cos^{t-\Delta}\Big(\frac{\pi}{2n}\Big)\sin\Big(\frac{\pi X_\Delta}{2n}\Big)\mid \tau_{\{x,2n\}}>\Delta\Big]\\
&=(1+\ord(n^{-2}))\frac{4}{\pi}\left(\frac{\cos\Big(\frac{\pi}{2n-x}\Big)}{\cos\Big(\frac{\pi}{2n}\Big)}\right)^{\Delta}\sin\Big(\frac{\pi(n-x)}{2n-x}\Big)\ex_n\Big[\sin\Big(\frac{\pi X_\Delta}{2n}\Big)\mid \tau_{\{x,2n\}}>\Delta\Big].
\end{align*}
Again using $\cos x=e^{-\frac{x^2}{2}}(1+\ord(x^4))$ we have
\begin{align}
\nonumber (1+&\ord(n^{-2})))\frac{4}{\pi}\left(\frac{\cos\Big(\frac{\pi}{2n-x}\Big)}{\cos\Big(\frac{\pi}{2n}\Big)}\right)^{\Delta}\sin\Big(\frac{\pi(n-x)}{2n-x}\Big)\ex_n\Big[\sin\Big(\frac{\pi X_\Delta}{2n}\Big)\mid \tau_{\{x,2n\}}>\Delta\Big]\\
\nonumber &=(1+\ord(\Delta n^{-4}))\frac{4}{\pi}e^{-{\Delta\pi^2({(2n-x)^{-2}}-{4n^{-2}})}/{2}}\sin\Big(\frac{\pi(n-x)}{2n-x}\Big)\ex_n\Big[\sin\Big(\frac{\pi X_\Delta}{2n}\Big)\mid \tau_{\{x,2n\}}>\Delta\Big]\\
\label{rela4}&=(1+\ord(\Delta n^{-4}))\frac{4}{\pi}e^{-{\Delta x\pi^2}/{(8n^3)}}\ex_n\Big[\sin\Big(\frac{\pi X_\Delta}{2n}\Big)\mid \tau_{\{x,2n\}}>\Delta\Big].
\end{align}
An important point here is that this probability asymptotically does not depend on~$t$, just on $n$, $\Delta$ and $x$. Now, to work with the expectation, we will show that its value is asymptotically equal to $\pi/4$, using for this Lemma~\ref{lemmapi4}. Consider the set $\Gamma_{n,\Delta}$ of all paths started in $n$ and of length $\Delta$, so we can write the following expectation in terms of a sum of probabilities of paths and use translation invariance of simple random walk:
\begin{align}
\nonumber \ex_n\Big[\sin\Big(\frac{\pi X_\Delta}{2n}\Big)\mid \tau_{\{x,2n\}}>\Delta\Big]&=\sum \limits_{\gamma \in \Gamma_{n,\Delta}} \sin\Big(\frac{\pi \gamma_\Delta}{2n}\Big) \pr_n[\gamma \mid \tau_{\{x,2n\}}>\Delta]\\
\nonumber &=\sum \limits_{\gamma \in \Gamma_{n,\Delta}} \sin\Big(\frac{\pi \gamma_\Delta}{2n}\Big) \frac{\pr_n[\gamma, \tau_{\{x,2n\}}>\Delta]}{\pr_n[\tau_{\{x,2n\}}>\Delta]}\\
\nonumber &=\sum \limits_{\gamma \in \Gamma_{n-x,\Delta}} \sin\Big(\frac{\pi (x+\gamma_\Delta)}{2n}\Big) \frac{\pr_{n-x}[\gamma, \tau_{\{0,2n-x\}}>\Delta]}{\pr_{n-x}[\tau_{\{0,2n-x\}}>\Delta]}\\
\label{rela3}&=\ex_{n-x}\Big[\sin\Big(\frac{\pi (x+X_\Delta)}{2n}\Big)\mid \tau_{\{0,2n-x\}}>\Delta\Big].
\end{align}
We can write
\begin{align}
\nonumber\ex_{n-x}&\Big[\sin\Big(\frac{\pi (x+X_\Delta)}{2n}\Big)\mid \tau_{\{0,2n-x\}}>\Delta\Big]\\
\nonumber &=\sin\Big(\frac{\pi x}{2n}\Big)\ex_{n-x}\Big[\cos\Big(\frac{\pi X_\Delta}{2n}\Big)\mid \tau_{\{0,2n-x\}}>\Delta\Big]\nonumber\\
& \qquad{}+\cos\Big(\frac{\pi x}{2n}\Big)\ex_{n-x}\Big[\sin\Big(\frac{\pi X_\Delta}{2n}\Big)\mid \tau_{\{0,2n-x\}}>\Delta\Big]\\
\label{rela}&=\ord(n^{-1})+(1-\ord(n^{-2}))\ex_{n-x}[\sin\Big(\frac{\pi X_\Delta}{2n}\Big)\mid \tau_{\{0,2n-x\}}>\Delta].
\end{align}  
Now working with the sine in the expectation
\begin{align*}
\ex_{n-x}&\Big[\sin\Big(\frac{\pi X_\Delta}{2n}\Big)\mid \tau_{\{0,2n-x\}}>\Delta\Big]=\ex_{n-x}\Big[\sin\Big(\frac{\pi X_\Delta}{2n-x}\Big(1-\frac{x}{2n}\Big)\Big)\mid \tau_{\{0,2n-x\}}>\Delta\Big]\\
&=\ex_{n-x}\Big[\sin\Big(\frac{\pi X_\Delta}{2n-x}\Big)\cos\Big(\frac{\pi x X_\Delta}{2n(2n-x)}\Big)\mid \tau_{\{0,2n-x\}}>\Delta\Big]\\
&\quad-\ex_{n-x}\Big[\cos\Big(\frac{\pi X_\Delta}{2n-x}\Big)\sin\Big(\frac{\pi x X_\Delta}{2n(2n-x)}\Big)\mid \tau_{\{0,2n-x\}}>\Delta\Big].
\end{align*}  
As $0<X_\Delta<2n-x$ we have the following asymptotic behavior,
\begin{align*}
\cos\Big(\frac{\pi x X_\Delta}{2n(2n-x)}\Big)&= 1-\ord \Big(\frac{(\pi x X_\Delta)^2}{2(2n(2n-x))^2}\Big)\\
&=1-\ord(n^{-2}),\intertext{and}
\sin\Big(\frac{\pi x X_\Delta}{2n(2n-x)}\Big)&=\ord\Big(\frac{\pi x X_\Delta}{2n(2n-x)}\Big)\\
&=\ord(n^{-1}).
 \end{align*} 
 So with this we get
 \begin{align}
 \nonumber\ex_{n-x}&\Big[\sin\Big(\frac{\pi X_\Delta}{2n}\Big)\mid \tau_{\{0,2n-x\}}>\Delta\Big]\\
\label{rela2}&=(1-\ord(n^{-2}))\ex_{n-x}\Big[\sin\Big(\frac{\pi X_\Delta}{2n-x}\Big)\mid \tau_{\{0,2n-x\}}>\Delta\Big]+\ord(n^{-1}).
\end{align}  
By Lemma~\ref{lemmapi4} we have
\[
\ex_{n-x}\Big[\sin\Big(\frac{\pi X_\Delta}{2n-x}\Big)\mid \tau_{\{0,2n-x\}}>\Delta\Big]=\frac{\pi}{4}(1+\ord(n^{-2})) .
\]
 Then using this in $\eqref{rela2}$ we get
  \begin{align}
 \nonumber\ex_{n-x}&\Big[\sin\Big(\frac{\pi X_\Delta}{2n}\Big)\mid \tau_{\{0,2n-x\}}>\Delta\Big]= \frac{\pi}{4}(1+\ord(n^{-1})).
\end{align}  
Finally, using this in $\eqref{rela}$ and then in $\eqref{rela3}$ we get 
 \begin{align}
\label{sinerelation} \ex_n\Big[\sin\Big(\frac{\pi X_\Delta}{2n}\Big)\mid \tau_{\{x,2n\}}>\Delta\Big]=(1+\ord(n^{-1})+\frac{\pi}{4}\ord(\Delta n^{-4})) .
\end{align}  
Using this in $\eqref{rela4}$, we conclude the proof of Lemma~\ref{lemmapiece}.
\end{proof}

\begin{lemma} Suppose $x>1$ is fixed and $y^2=o(\Delta)$. It holds that
\label{lemmaend}
\begin{align*}
\cpr_x[X_\Delta\leq y]=\pr_x[X_\Delta\leq y\mid \tau_0=\infty]=\sqrt{\frac{2}{\pi}}\frac{y^{3}}{3\Delta^{{3}/{2}}}(1+\ord(y^2\Delta^{-1})).
\end{align*}

\end{lemma}
\begin{proof} 
Let us calculate this probability by splitting it according to the endpoint of the paths:
\begin{align}
\label{cprend}\cpr_x[X_\Delta\leq y]=\sum \limits_{k=1}^y \frac{k}{x2^\Delta}n(\gamma:\gamma_0=x,\, \gamma_i\neq 0,\, \gamma_\Delta=y).
\end{align}
We need to estimate $N_k:=n(\gamma:\gamma_0=x,\, \gamma_i\neq 0,\, \gamma_\Delta=k)$.
In this sum we have some problems with parity. If $\Delta$ is even, then both $x$ and $\gamma_\Delta$ need to have the same parity. Otherwise they need to have opposite parity. As we are mostly interested in asymptotic results, let us assume that $\Delta$  even and so $x$ and $k$ are of the same parity.

Here we have paths from $(0,x)$ to $(\Delta, y)$  that do not touch the line $y=0$. These value can be explicitly calculated using the reflection principle (see section~1 of chapter~III of~\cite{feller}), so we have:
\begin{align*}
N_k&=\binom{\Delta}{\frac{\Delta+k-x}{2}}-\binom{\Delta}{\frac{\Delta-x-k}{2}}\\
&=\binom{\Delta}{\frac{\Delta+k-x}{2}}\left(1-\frac{\Big(\frac{\Delta+k-x}{2}\Big)!\Big(\frac{\Delta-k+x}{2}\Big)!}{\Big(\frac{\Delta+k+x}{2}\Big)!\Big(\frac{\Delta-k-x}{2}\Big)!}\right).
\end{align*}
We have two fractions to work with, so as $\Delta+k$ goes to infinity, we use the asymptotic expansions valid for any real constant $a$:
\begin{align*}
\frac{\Gamma(n+a+1)}{\Gamma(n-a+1)}=n^{2a}\Big(1+\frac{a}{n}+\ord(n^{-2})\Big).
\end{align*}
Then 
\begin{align}
\nonumber N_k&=\binom{\Delta}{\frac{\Delta+k-x}{2}}\left(1-\Big(\frac{\Delta-k}{\Delta+k}\Big)^{x}\Big(1-\frac{x}{\Delta+k}+\ord(\Delta^{-2})\Big)\right)\\
\nonumber &=\binom{\Delta}{\frac{\Delta+k-x}{2}}\left(1-\Big(1-\frac{2kx}{\Delta+k}+\ord(k^2\Delta^{-2})\Big)\Big(1-\frac{x}{\Delta+k}+\ord(\Delta^{-2})\Big)\right)\\
\label{factorial1}&=\binom{\Delta}{\frac{\Delta+k-x}{2}}\left(\frac{(2k+1)x}{\Delta}+\ord(k^2\Delta^{-2})\right).
\end{align}
Using the Stirling approximation $n!=\sqrt{2\pi n} (n/e)^n(1+\ord(n^{-1}))$ we can work with the asymptotic expression of the remaining binomial term:
\begin{align*}
\binom{\Delta}{\frac{\Delta+k-x}{2}}&=\frac{\Delta!}{\frac{\Delta+k-x}{2}!\frac{\Delta-k+x}{2}!}\\
&=\frac{\sqrt{2\pi \Delta}\Big(\frac{\Delta}{e}\Big)^{\Delta}}{\pi\sqrt{ \Delta^2-(k-x)^2}\Big(\frac{\Delta+k-x}{2e}\Big)^{{(\Delta+k-x)}/{2}}\Big(\frac{\Delta-k+x}{2e}\Big)^{{(\Delta-k+x)}/{2}}}(1+\ord(\Delta^{-1}))\\
&=2^{\Delta}\sqrt{\frac{2}{\pi\Delta}}(1+\ord(k^2\Delta^{-1})).
\end{align*}
Now using this in $\eqref{factorial1}$ we get
\begin{align}
\label{nk} N_k&=2^{\Delta}\sqrt{\frac{2}{\pi}}\frac{(2k+1)x}{\Delta^{{3}/{2}}}(1+\ord(k^2\Delta^{-1})).
\end{align}
We want to use $\eqref{nk}$  in $\eqref{cprend}$, but we need to worry about the parity before. First, if $x$ is even we get
\begin{align}
\nonumber \cpr_x[X_\Delta\leq y]&=\sqrt{\frac{2}{\pi}}\frac{x}{x\Delta^{{3}/{2}}}\sum \limits_{k=1}^{\lfloor y/2\rfloor} {2k(4k+1)}(1+\ord(k^2\Delta^{-1}))\\
\label{result1} &=\sqrt{\frac{2}{\pi}}\frac{y^{3}}{3\Delta^{{3}/{2}}}(1+\ord(y^2\Delta^{-1})).
\end{align}
It is straightforward to see that $\cpr_x[X_\Delta\leq y]$ is decreasing in $x$. This monotonicity property allow us to extend $\eqref{result1}$ to any value of $x$. The same argument can be used for $\Delta$, as the positive drift makes $\cpr_x[X_\Delta\leq y]$ also decreasing in $\Delta$. This makes this asymptotic expression valid for all $x$, $y$, and $\Delta$ satisfying the condition in the Lemma, which concludes the proof of Lemma~\ref{lemmaend}.
\end{proof}

\begin{coro} For a fixed $x>0$, consider quantities $\Delta$ and $y$  possibly depending on~$n$ such that $\Delta=o(n^2)$, $y^2=o(\Delta)$, $y=o(n)$ and $t$ satisfy Condition $\eqref{cond}$. Then 
\label{coroend}
\begin{align*}
\rpr^t_x[X_\Delta\leq y]=(1+\ord(\Delta n^{-2})+\ord(y^2\Delta^{-1}))\sqrt{\frac{2}{\pi}}\frac{y^{3}}{3\Delta^{{3}/{2}}}.
\end{align*}

\end{coro}

\begin{proof} Consider the set $\Gamma$ of paths $\gamma$ of length $\Delta$ with the property: $\gamma_0=x$, $\gamma_i \notin \{0,n\}$ for all $i$ and $\gamma_\Delta \leq y$. Then, using Lemma~\ref{lemmaend}
\begin{align*}
\rpr^t_x[X_\Delta\leq y]&=\sum \limits_{\gamma \in \Gamma} \rpr^t_x[\gamma]\\
&=\sum \limits_{\gamma \in \Gamma} \frac{h_n(\gamma_\Delta,t-\Delta)}{2^\Delta h_n(x,t)}\\
&=(1+\ord(n^{-2}))\frac{\cos^{-\Delta}\Big(\frac{\pi }{n}\Big)}{\sin\Big(\frac{\pi x}{n}\Big)}\sum \limits_{\gamma \in \Gamma} \frac{\sin\Big(\frac{\pi \gamma_\Delta}{n}\Big)}{2^\Delta }\\
&=(1+\ord(y^2 n^{-2})){\cos^{-\Delta}\Big(\frac{\pi }{n}\Big)}\sum \limits_{\gamma \in \Gamma} \frac{\gamma_\Delta}{x2^\Delta }\\
&=(1+\ord(\Delta n^{-2})+\ord(y^2 n^{-2}))\cpr_x[X_\Delta \leq y]\\
&=(1+\ord(\Delta n^{-2})+\ord(y^2\Delta^{-1}))\sqrt{\frac{2}{\pi}}\frac{y^{3}}{3\Delta^{{3}/{2}}}.
\end{align*}
This concludes the proof of Corollary~\ref{coroend}.
\end{proof}

 Now we prove the convergence of the local time for a fixed $x$. First we present a sketch of the proof:
\begin{itemize}
\item We divide our time interval o length $t^*$ into $m$ random intervals. This is done in such a way that these time intervals have roughly the same length.
\item For each of these intervals we have almost equal chances that the conditional random walk hits $x$. When the walk hits $x$ we start coupling an excursion with the conditional random walk on $\zp$ for a fixed time $T$. 
\end{itemize}

Here we are interested in coupling of the local times, so our coupling procedure will fail if any visit to~$x$ is not in the coupled excursions or if the coupling of the excursions fails.

Now in the proof of Theorem~\ref{teolocal} we will rigorously define the terms used in the sketch.
\begin{proof}[Proof of Theorem \ref{teolocal}]

We begin by considering the ring graph of size~$2n$ with the conditional random walk started at site~$n$. The vacant set convergence time is~$t^*=4\alpha n^3/\pi^2$ by Theorem~\ref{teolimit}. We will split this interval in  $m=\lfloor\ln n\rfloor$ random intervals. For this purpose let us define the sequence of points $A_j$ in the following way: Let $\eta=\left\lfloor t^*/\ln n\right\rfloor$ and with it define
\begin{align*}
A_0&=0;\\
A_{j+1}&=\inf \{t\geq A_j+\eta : X_t=n  \},\quad \mbox{ if } j<m;\\
A_{m+1}&=t^*.
\end{align*}
Then, the interval $I_j$ for $j\leq m$ is defined as
\begin{equation}
\label{interval} I_j:=[A_{j-1},A_j),
\end{equation}
and the remaining time interval~$R$ is defined as
\begin{equation}
\label{remaining}R:=[A_m, A_{m+1}].
\end{equation}
Now we  want to show that, almost surely, the lengths of all above intervals are asymptotic to $\eta$; by definition of $I_j$ we can say that its length can be represented as $I_j=\eta+T_j$, where $T_j$  satisfies
\begin{equation}
\label{lawT} \pr[T_j=a]=\ex_n\big[\rpr^{t^*-A_{j-1}-\eta}_{X_\eta}[\tau_n=a] \mid \tau_{\{0,2n\}}>t^*-A_{j-1}\big].
\end{equation}
This is because from the moment $A_{j-1}$ the process still have $t^*-A_{j-1}$ steps to run without hitting the origin. Also, we want the first time where $X_t=n$ after $\eta$, so the starting point is $X_\eta$ and from there are considering the hitting time of $n$, which justifies the expression for the probability inside the expectation.

Consider $j<m$, for any $\epsilon>0$ we have
\begin{align}
\nonumber\pr[|I_j|>(1+\epsilon)\eta]&=\pr_[T_j>\epsilon\eta]\\
\label{Tbound}&=\ex_n[\rpr^{t^*-A_{j-1}-\eta}_{X_\eta}[\tau_n>\epsilon\eta] \mid \tau_{\{0,2n\}}>t^*-A_{j-1}].
\end{align}
Observe also that $\epsilon\eta$, $t^*-A_{j-1}-\eta$, and the difference $t^*-A_{j-1}-(1+\epsilon)\eta$ all satisfy condition $\eqref{cond}$, so we use Lemma~\ref{lemmamid} and get
\begin{align*}
\pr[|I_j|>(1+\epsilon)\eta]&\leq (1+o(1))\frac{8}{\pi}\exp\big\{-\frac{3\pi^2\epsilon\eta}{8n^2}\Big\}\ex_n\Big[ \cos\Big(\frac{\pi X_\eta}{2n}\Big) \mid \tau_{\{0,2n\}}>t^*-A_{j-1}\Big]\\
&\leq (1+o(1))\frac{8}{\pi}\exp\big\{-\frac{3\pi^2\epsilon\eta}{8n^2}\Big\}\\
&=(1+o(1))\frac{8}{\pi}\exp\Big\{-\frac{3\alpha\epsilon n}{2}\Big\}.
\end{align*}
So, we have  an exponential bound for the tail probability which is summable, and therefore this shows that $I_j\sim \eta$ a.s..

Before constructing the coupling, let us discuss the probability of a successful coupling between trajectories of conditional random walks on the ring graph up to time $t$ and on $\zp$. Assume that $t$ is of order $n^3$ and let $x$ be a fixed value, we are interested in coupling the paths of both processes up to a time $T=n^\mu$ (where $\mu<1$), using the maximal coupling. Then the coupling event probability $\pr[C]$ can be estimated using the expressions for the laws~$\eqref{lawring}$ and $\eqref{lawcond}$. Let $\Gamma$ be the set of all paths started in $x$ and with length $T$. 
We have the following expression for the coupling event probability
\begin{align*}
\pr[C^\complement]&=\frac{1}{2} \sum \limits_{\gamma \in \Gamma} \big\vert \cpr_x[\gamma]-\rpr_x^t[\gamma]\big\vert.
\end{align*}
We use Theorem~\ref{asymp}, but we have the stronger condition that~$t$ is of order~$n^3$, so instead we use expression $\eqref{exact}$ inside the proof and for each term in the sum we have 
\begin{align*}
\big\vert \cpr_x[\gamma]-\rpr_x^t[\gamma]\big\vert&= \left\vert \frac{\gamma_T}{2^T x}-\frac{h_{2n}(\gamma_T,t-T)}{2^T h_{2n}(x,t)}\right\vert\\
&= \frac{1}{2^T}\left\vert \frac{\gamma_T}{x}-(1+\ord(n e^{-{3t\pi^2}/{(2n^2)}}))\cos^{-T}\Big(\frac{\pi}{2n}\Big)\frac{\sin\Big(\frac{\pi \gamma_T}{2n}\Big)}{\sin\Big(\frac{\pi x}{2n}\Big)}\right\vert.
\end{align*}
As $\cos^{-T}\Big(\frac{\pi}{2n}\Big)=1-\ord(n^{-2}T)$, the error term in the cosine asymptotic approximation dominates the error term in  expression $\eqref{exact}$ and then 
\[
(1+\ord(n e^{-{3t\pi^2}/{(2n^2)}}))\cos^{-T}\Big(\frac{\pi}{2n}\Big)=1-\ord(n^{-2}T).
\]
So, we have
\begin{align}
\nonumber \big\vert \cpr_x[\gamma]-\rpr_x^t[\gamma]\big\vert&=\frac{1}{2^T} \left\vert \frac{\gamma_T}{x}-(1-\ord(n^{-2}T))\frac{\sin\Big(\frac{\pi \gamma_T}{2n}\Big)}{\sin\Big(\frac{\pi x}{2n}\Big)}\right\vert\\
\label{2terms}&\leq\frac{1}{2^T} \left\vert \frac{\gamma_T}{x}-\frac{\sin\Big(\frac{\pi \gamma_T}{2n}\Big)}{\sin\Big(\frac{\pi x}{2n}\Big)}\right\vert+\ord(n^{-2}T)\frac{\sin\Big(\frac{\pi \gamma_T}{2n}\Big)}{2^T\sin\Big(\frac{\pi x}{2n}\Big)}.
\end{align}
Next,  we sum the second term of $\eqref{2terms}$ in $\Gamma$ to obtain
\begin{align*}
\sum \limits_{\gamma \in \Gamma}\frac{\sin\Big(\frac{\pi \gamma_T}{2n}\Big)}{2^T\sin\Big(\frac{\pi x}{2n}\Big)}&=\frac{\ex_x\Big[\sin\Big(\frac{\pi X_T}{2n}\Big)\ind[\tau_0>T]\Big]}{\sin\Big(\frac{\pi x}{2n}\Big)}\\
&\leq\frac{\ex_x\sin\Big(\frac{\pi X_T}{2n}\Big)}{\sin\Big(\frac{\pi x}{2n}\Big)}
\end{align*}
As  the sine is concave on the interval $[0,\pi]$ we can use Jensen's inequality and get 
\begin{align}
\nonumber \frac{\ex_x\sin\Big(\frac{\pi X_T}{2n}\Big)}{\sin\Big(\frac{\pi x}{2n}\Big)}&\leq  \frac{\ex_x\sin\Big(\frac{\pi X_T}{2n}\Big)}{\sin\Big(\frac{\pi x}{2n}\Big)}\\
\label{back1}&\leq \frac{\sin\Big(\frac{\pi \ex_xX_T}{2n}\Big)}{\sin\Big(\frac{\pi x}{2n}\Big)}=1.
\end{align}
Now we sum the first term of~$\eqref{2terms}$ in $\Gamma$ to get
\begin{align*}
\sum \limits_{\gamma \in \Gamma}\frac{1}{2^T} \left\vert \frac{\gamma_T}{x}-\frac{\sin\Big(\frac{\pi \gamma_T}{2n}\Big)}{\sin\Big(\frac{\pi x}{2n}\Big)}\right\vert&\leq\ex_x\left\vert \frac{X_T}{x}-\frac{\sin\Big(\frac{\pi X_T}{2n}\Big)}{\sin\Big(\frac{\pi x}{2n}\Big)}\right\vert.
\end{align*}
As the maximum value that $X_T$ can achieve starting at $x$ is $x+T$, we have that $X_T=o(n)$ and then we can use the asymptotic expression for the  sine and get
\begin{align}
\nonumber \ex_x\left\vert \frac{X_T}{x}-\frac{\sin\Big(\frac{\pi X_T}{2n}\Big)}{\sin\Big(\frac{\pi x}{2n}\Big)} \right \vert &=\ex_x\left\vert \frac{X_T}{x}-\frac{X_T-\ord(X_T^3 n^{-2} )}{x-\ord(n^{-2})} \right \vert\\
\nonumber &=\ex_x\left\vert \frac{X_T}{x}\ord(X_T^2 n^{-2} )\right \vert\\
\label{back2}&=\ord( n^{-2} )\ex_x\left\vert {X_T}\right \vert^3.
\end{align}
By the fact that $X_T$ is a simple random walk, we can represent the steps as a sequence of i.i.d. random variables $Y_i$, where $Y_i$ take values $+1$ and $-1$ with probability $1/2$, so $X_t=x+S_t$, where $S_t=Y_1+Y_2+\ldots+Y_t$. We will use this to bound $\ex_x\left\vert {X_T}\right \vert^3$.
\begin{align*}
\ex_x\left\vert {X_T}\right \vert^3&=\ex\left\vert {x+S_t}\right \vert^3\\
&\leq \ex\left\vert S_T\right \vert^3+3x\ex_x\left\vert S_T\right \vert^2+\ord(T^{1/2})\\
&=\ex\left\vert S_T\right \vert^3+\ord(T)
\end{align*}
Theorem 7.1.1 of~\cite{probM} shows that for the simple random 
walk~$S_t$ started at~$0$ and positive~$x$ it holds
\[
\pr[S_t\geq x]<\exp\Big\{-\frac{t^2}{2t}  \Big\}.
\]
For any constant $\beta>0$ we have
\begin{align*}
\pr[\left\vert S_T\right \vert^3\geq\beta T^2]&=2\pr[S_T\geq\beta^{\frac{1}{3}} T^{\frac{2}{3}}]\\
&\leq 2\exp\Big\{-\frac{\beta^{\frac{2}{3}}T^{\frac{1}{3}}}{2} \Big\},
\end{align*}
 a stretched exponential bound in $T$.

With this we can finally estimate $\ex_x\left\vert S_T\right \vert^3$. Let $D$ be the event $\{\left\vert S_T\right \vert^3\geq\beta T^2\}$, then, since $\left\vert S_T\right \vert\leq T$, we have
\begin{align*}
\ex\left\vert S_T\right \vert^3&=\pr[D]\ex[\left\vert S_T\right \vert^3\mid D]+\pr[D^C]\ex[\left\vert S_T\right \vert^3\mid D^C]\\
&\leq2\exp\Big\{-\frac{\beta^{{2}/{3}}T^{{1}/{3}}}{2}\Big\}{T^3}+\beta T^2=\ord(T^2).
\end{align*}
Finally,
\begin{align}
\nonumber \ex_x\left\vert {X_T}^3\right \vert&\leq \ex\left\vert S_T\right \vert^3+\ord(T)\\
\label{back3}&=\ord(T^2).
\end{align}
Then we can bound the coupling event probability. Using $\eqref{back1}$, $\eqref{back2}$ and $\eqref{back3}$ together we have
\begin{align}
\nonumber 2\pr[C^\complement]&=\sum \limits_{\gamma \in \Gamma} \big\vert \cpr_x[\gamma]-\rpr_x^t[\gamma]\big\vert\\
\label{f2}&=\ord(n^{-2})\ord(T^2)+\ord(n^{-2}T)=\ord(n^{-2}T^2).
\end{align}

We are still missing the probability that our procedure fails because the random walk has at least two excursions in at least one of intervals. 

Let us consider excursions of length $T$ starting at $x$. The initial times of the first and (possibly) second excursion starting in $x$ are denoted by
\begin{align*}
\tau^1_x(k)&=\inf \{t\in I_k: X_t=x\},\\
\tau^2_x(k)&=\inf \{t>\tau^1_x(k)+T: X_t=x\}.
\end{align*}
Observe that we do not necessarily have $\tau^2_x(k) \in I_k$, so the event that two or more excursions happen during $I_k$ is $\{\tau^2_x(k)\in I_k\}$. We want to calculate the probability of this event. The initial point of interval $I_k$ is $A_{k-1}$, so the process still have time $t^*-A_{k-1}$ to run. So, using the Markov property, we have
\begin{align}
\nonumber \lefteqn{\rpr_n^{t^*-A_{k-1}}[\tau^2_x(k)\in I_k]}\\
\label{terms}&=\rpr_n^{t^*-A_{k-1}}[\tau^1_x(k)\in I_k]\cdot\rpr_n^{t^*-A_{k-1}}[\tau^2_x(k)\in I_k\mid \tau^1_x(k)\in I_k].
\end{align}
Now let us work with each term of $\eqref{terms}$ separately. As both $|I_k|$ and $t^*-A_{k-1}-|I_k|$ satisfy almost surely condition $\eqref{cond}$, we can use Lemma~\ref{lemmapiece} and the fact that $|I_k|$ is of order $n^3(\ln n)^{-1}$, then we get
\begin{align*}
\rpr_n^{t^*-A_{k-1}}[\tau^1_x(k)\in I_k]&=\pr_n[\tau_x\leq |I_k| \mid\tau_{\{0,2n\}}>{t^*-A_{k-1}}]\\
&=\ex \Big[\ex\big[\pr_n[\tau_x\leq |I_k| \mid\tau_{\{0,2n\}}>{t^*-A_{k-1}}]\mid \sigma(I_k)\big]\Big]\\
&=1-\ex \big[\ex [(1+\ord(n^{-1})+\ord(  |I_k|n^{-4}))e^{-{|I_k|\pi^2 x}/{(8n^3)}}\mid \sigma(I_k)]\big]\\
&=1-(1+\ord(n^{-1}))\ex e^{-{|I_k|\pi^2 x}/{(8n^3)}}\\
&=1-(1+\ord(n^{-1}))e^{-{\eta\pi^2 x}/{(8n^3)}}\ex e^{-{T_k\pi^2 x}/{(8n^3)}}.
\end{align*}
Using Lemma~\ref{lemmamid} in the same way as we did in $\eqref{Tbound}$ we have:
\begin{align*}
\ex e^{-{T_k\pi^2 x}/{(8n^3)}}&=\pr[T_k>4 n^2 \ln n/\pi^2]\ex [e^{-{T_k\pi^2 x}/{(8n^3)}}\mid T_k> 4n^2 \ln n/\pi^2]\\
&\quad+\pr[T_k\leq 4n^2 \ln n/\pi^2]\ex [e^{-{T_k\pi^2 x}/{(8n^3)}}\mid T_k\leq 4n^2 \ln n/\pi^2]\\
&=\ord(n^{-\frac{3}{2}})+(1-\ord(n^{-\frac{3}{2}}))]\ex [e^{-{T_k\pi^2 x}/{(8n^3)}}\mid T_k\leq 4n^2 \ln n/\pi^2]\\
&=\ord(n^{-\frac{3}{2}})+(1-\ord(n^{-\frac{3}{2}}))(1-\ord(n^{-1}\ln n))\\
&=1-\ord(n^{-1}\ln n),
\end{align*}
and therefore
\begin{align*}
\rpr_n^{t^*-A_{k-1}}[\tau^1_x(k)\in I_k]&=1-(1+\ord(n^{-1}\ln n))e^{-{\eta \pi^2 x}/{(8n^3)}}\\
&=1-(1+\ord(n^{-1}\ln n))e^{-{\alpha x}/{(2\ln n)}}.
\end{align*}
As $e^{-{\alpha x}/{(2\ln n)}}=1-\frac{\alpha x}{2\ln n}+\ord((\ln n)^{-2})$, we obtain 

\begin{align}
\nonumber \rpr_n^{t^*-A_{k-1}}[\tau^1_x(k)\in I_k]&=1-(1+\ord(n^{-1}\ln n))\Big(1-\frac{\alpha x}{2\ln n}+\ord((\ln n)^{-2})\Big)\\
\label{prob} &=\frac{\alpha x}{2\ln n}+\ord(n^{-1}\ln n).
\end{align}
Now we work with the second probability of $\eqref{terms}$:
\begin{align*}
\rpr_n^{t^*-A_{k-1}}&[\tau^2_x(k)\in I_k\mid \tau^1_x(k)\in I_k]\\
&=\rex^{t^*-A_{k-1}}_x \pr_{X_T}[\tau_x\leq |I_k|-\tau_x^1(k)-T\mid\tau_{\{0,2n\}}> t^*-P_{k-1}-\tau^1_x(k)]
\end{align*}
Let us abbreviate $t^{**}={t^*-A_{k-1}}$. Again we use Corollary~\ref{coroend} and Lemma~\ref{lemmapiece}; then, as $T=n^\mu$ we get
\begin{align*}
\rex_x^{t^{**}} &\pr_{X_T}[\tau_x\leq |I_k|-\tau_x^1(k)-T\mid\tau_{\{0,2n\}}> t^*-A_{k-1}-\tau^1_x(k)]\\
&=\rpr_x[X_T\leq n^{\frac{\mu}{3}}]\rex_x^{t^{**}} \pr_{X_T}[\tau_x\leq |I_k|-\tau_x^1(k)-T\mid\tau_{\{0,2n\}}> t^*-P_{k-1}-\tau^1_x(k)]\mid X_T\leq n^{\frac{\mu}{3}}]\\
&\quad+\rpr_x[X_T> n^{\frac{\mu}{3}}]\rex_x^{t^{**}} \pr_{X_T}[\tau_x\leq |I_k|-\tau_x^1(k)-T\mid\tau_{\{0,2n\}}> t^*-P_{k-1}-\tau^1_x(k)]\mid X_T>n^{\frac{\mu}{3}}]\\
&=\ord(n^{-\frac{\mu}{2}})+(1-\ord(n^{-\frac{\mu}{2}}))\Big(1-  (1+\ord(n^{-1}))\rex_x^{t^{**}}e^{-{(|I_k|-\tau_x^1(k)-T)x\pi^2}/{(8n^3)}}\Big)\\
&\leq \ord(n^{-\frac{\mu}{2}})+(1-\ord(n^{-\frac{\mu}{2}}))\Big(1-  (1+\ord(n^{-1}\ln n))\rex_x^{t^{**}}e^{-{|I_k|x\pi^2}/{(8n^3)}}\Big)\\
&=\ord(n^{-\frac{\mu}{2}})+(1-\ord(n^{-\frac{\mu}{2}}))\Big(1-  (1+\ord(n^{-1}\ln n))e^{-{\alpha x}/{(2\ln n)}}\Big)\\
&=\ord(n^{-\frac{\mu}{2}})+(1-\ord(n^{-\frac{\mu}{2}}))\Big(1-  (1+\ord(n^{-1}\ln n))\Big({1-\frac{\alpha x}{2\ln n}}\Big)\Big)\\
&=\frac{\alpha x}{2\ln n}(1+\ord(n^{-\frac{\mu}{2}}\ln n)).
\end{align*}
So, we can bound the probability that a specific interval $I_j$ contains at least two excursions to $x$:
\begin{align*}
\rpr_n^{t^*-A_{k-1}}[\tau^2_x(k)\in I_k]&\leq\Big(\frac{\alpha x}{2\ln n}+\ord(n^{-1}\ln n)\Big)\Big(\frac{\alpha x}{2\ln n}+\ord(n^{-\frac{\mu}{2}})\Big)\\
&=\frac{\alpha^2 x^2}{4\ln^2 n}+\ord(n^{-\frac{\mu}{2}}(\ln n)^{-1}).
\end{align*}
Finally we bound the probability that at least one interval contain at least two excursions
\begin{align}
\nonumber\rpr^{t^*}_n\Big[\bigcup_{k=1}^m\{\tau^2_x(k)\in I_k\}\Big]&\leq \sum \limits_{k=1}^m\rpr_n^{t^*-A_{k-1}}[\tau^2_x(k)\in I_k]\\
\nonumber&\leq m\Big(\frac{\alpha^2 x^2}{4\ln^2 n}+\ord(n^{-\frac{\mu}{2}}(\ln n)^{-1})\Big)\\
\label{f1}&=\frac{\alpha^2 x^2}{4\ln n}+\ord((\ln n)^{-2})\to 0\quad\mbox{as }n\to\infty.
\end{align}

After this we only need to worry about the remaining time $R$. We surely have that it is smaller than any of the intervals, as for any $j\leq m$ we have $|I_j|\geq \eta$, but $(m+1)\eta>t^*$. Let us bound the probability that the process visits~$x$ in the remaining time. 

We want to estimate $\rpr^R_{n}[\tau_x>R]$; consider the function $f:\zp \rightarrow [0,1]$ defined by
\[
f(t)=\rpr^t_n[\tau_x>t]
\]
By definition we have
\begin{align*}
\rpr^t_n[\tau_x>t]&=\frac{\pr_n[\tau_{\{x,2n\}}>t]}{\pr_n[\tau_{\{0,2n\}}>t}\\
&=\frac{h_{2n-x}(n-x,t)}{h_{2n}(n,t)}.
\end{align*}
If $t$ satisfies condition~$\eqref{cond}$ we have an asymptotic expression for $f(t)$
\begin{align*}
f(t)&=(1+\ord(n^{-2})){\sin\Big(\frac{\pi (n-x)}{2n-x}\Big)}\frac{\cos^{t}\Big(\frac{\pi }{2n-x}\Big)}{\cos^{t}\Big(\frac{\pi }{2n}\Big)}\\
&=(1+\ord(tn^{-4})))\exp\Big\{-\frac{t\pi^2x}{8n^3}\Big\}.
\end{align*}
So, asymptotically the function is decreasing. We now show that $R$ satisfies condition $\eqref{cond}$. Fix a constant $\beta >4/\pi^2$, then write
\begin{align}
\nonumber\pr[R>\beta n^2 \ln n]&=\pr\Big[t^*-m\eta-\sum_{k=1}^m T_k>\beta n^2 \ln n\Big]\\
\nonumber &=\pr\Big[\sum_{k=1}^m T_k<t^*-m\eta-\beta n^2 \ln n\Big]\\
\label{bounding}&\geq 1-\frac{\sum_{k=1}^m \ex T_k}{t^*-m\beta\eta n^2 \ln n}.
\end{align}
To bound the expectations in $\eqref{bounding}$ we use $\eqref{lawT}$:
\begin{align*}
\pr\Big[T_k>\frac{16n^2\ln n}{\pi^2}\Big]&=\rex_n\Big[\rpr_{X_\eta}^{t^*-A_{k-1}-\eta}\Big[\tau_n>\frac{16n^2\ln n}{\pi^2}\Big]\mid \tau_{\{0,2n\}}>t^*-A_{k-1}\Big]\\
&= (1+\ord(n^{-2}\ln n))\frac{8}{\pi}n^{-6}.
\end{align*}
Also, as each $T_k$ is bounded by $t^*$, we have
\begin{align*}
\ex T_k&=\ex\Big[ T_k\mid T_k>\frac{16n^2\ln n}{\pi^2}\Big]\pr\Big[ T_k>\frac{16n^2\ln n}{\pi^2}\Big]\\
&\quad+\ex\Big[ T_k\mid T_k\leq\frac{16n^2\ln n}{\pi^2}\Big]\pr\Big[ T_k\leq\frac{16n^2\ln n}{\pi^2}\Big]\\
&\leq \ord(t^*n^{-6})+\frac{16n^2\ln n}{\pi^2}(1+\ord(n^{-6}))\\
&\leq \frac{16n^2\ln n}{\pi^2}+ \ord(n^{-3}).
\end{align*}
Now observe that
\begin{align*}
t^*-m\eta&=t^*-(\lfloor \ln n \rfloor)\left\lfloor \frac{t^*}{\ln n}\right\rfloor\\
&=\ord({t^*}({\ln n})^{-1}).
\end{align*}
Therefore
\begin{align*}
\pr[R>\beta n^2 \ln n]&\geq 1-\frac{\sum_{k=1}^m \ex T_k}{t^*-m\eta-\beta n^2 \ln n}\\
&\geq 1-\frac{m\Big({16n^2\ln n}/{\pi^2}+ \ord(n^{-3})\Big)}{\ord({t^*}({\ln n})^{-1})}\\
&= 1-\ord((\ln n)^3n^{-1}).
\end{align*}
Then with high probability $R$ satisfies condition $\eqref{cond}$, consequently
\begin{align}
\nonumber \rpr^R_{n}[\tau_x>R]&=\ex f(R)\\
\nonumber &\geq \ex [f(R)\mid R>\beta n^2 \ln n]\pr[R>\beta n^2 \ln n]\\
\nonumber &=\ex [(1+\ord(Rn^{-4})))e^{-{R\pi^2x}/{(8n^3)}}\mid R>\beta n^2 \ln n](1-\ord((\ln n)^3n^{-1}))\\
\nonumber &=e^{-{\beta \ln n\pi^2x}/{(8n^3)}}(1-\ord((\ln n)^3n^{-1}))\\
\label{remainingvisit}&=1-\ord((\ln n)^3n^{-1}).
\end{align}
So, we have that $\rpr^R_{n}[\tau_x\leq R]=\ord((\ln n)^3n^{-1})$, which is an upper bound for the probability of a visit in the remaining time.

Now we construct the coupling. The motivation behind the procedure is that all visits to~$x$ usually happen in ``batches'',  so when  a initial visit to $x$ happens, the walk visits $x$ again at some moments during a small time interval and then goes away again. When trying to couple the entire process we get a error of large order that turns the coupling almost impossible to happen, so, as we are only interested in the visits to $x$, we just need to couple then for these small time intervals of the excursions. 

As $x$ is fixed and we are working with asymptotic behaviors, these ``batches'' of visits are rare and, since we splitted our time in the intervals defined in~\eqref{interval} such that in each of them  probability of hitting $x$ is almost the same; this makes our calculations possible. So again let us recall our definitions  and the coupling procedure:
\begin{itemize}
\item Consider a conditional random walk on the graph of length $2n$ that will run for a time $t^*=4\alpha n^3/\pi^2$. We split the time into~$m$ intervals as defined in $\eqref{interval}$ and the remaining time.
\item In each interval we have a small chance of visiting~$x$. When the visit happens, from that moment on we couple the walk on the ring with   the conditional random walk on $\zp$ for a time $T=n^{\mu}$.
\end{itemize}
Our procedure can fail if and only if any of the following happens:
\begin{itemize}
\item At least one of the $m$ intervals has 2 or more excursions.
\item There is a visit to~$x$ in the remaining time.
\item The maximal coupling fails for at least one excursion.
\end{itemize}
Let us denote these three events by $F_1$, $F_2$  and $F_3$  respectively, and the event that the coupling fails by~$F$, so $F=F_1\cup F_2\cup F_3$. Also, denote the local time of this walk in site $x$ by $L_n(x)$.

The probability of $F_1$ was already bounded in $\eqref{f1}$ and we bounded the probability of $F_3$ in $\eqref{remainingvisit}$. As for $F_2$, the probability that the coupling of one excursion fails was dealt with in $\eqref{f2}$; since as we have at most $m$ excursions, then
\begin{align*}
\pr[F_2]\leq m\ord(n^{-2}T^2)=\ord(n^{-2(1-\mu)}\ln n ).
\end{align*}
Consequently,
\begin{align*}
\pr[F]&\leq \pr[F_1]+\pr[ F_2]+\pr[ F_3]\\
&\leq\frac{\alpha^2 x^2}{4\ln n}+\ord((\ln n)^{-2})+\ord(n^{-2(1-\mu)}\ln n )+\ord((\ln n)^3n^{-1})\\
&=\frac{\alpha^2 x^2}{4\ln n}+\ord((\ln n)^{-2}).
\end{align*}

Now we have an estimate on the probability that the procedure fails; we need then to see which is the distribution of the number of visits to $x$ of an excursion. We have a conditional random walk on $\zp$ running up to time $n^\mu$. Let us consider each excursion as a part of a random walk in $\zp$ started in $x$ and denote by $V_i$ the number of visits of  $i$th random walk to $x$ if we let it run indefinitely; also denote by $T_{i,k}$ the time between the $k$th and $(k+1)$th visits. With this the number of visits of $i$th excursion $W_i$ can be defined as the number of visits of the walk up to time $n^\mu$ and be represented by
\[
W_i=\sum \limits_{j=1}^{V_i}\ind\left[\sum \limits_{k=1}^{j-1} T_{i,k}\leq n^\mu\right].
\]
Since $\sum\limits_{k=1}^{V_i-1} T_{i,k}$ is the time of the last visit; by Lemma~\ref{lemmacrw} it is finite as the number of visits is finite. This means that $W_i \rightarrow V_i$ almost surely, and since $|e^{itW_i}|=1$,  by the dominated convergence theorem
\begin{align*}
\phi_{W_i}(t)&=\ex e^{itW_i}\rightarrow \phi_{V_i}(t).
\end{align*}

Each coupled excursion is independent, so if we denote by $B_i$ the Bernoulli random variable indicating if the visit in $x$ at interval $I_i$ happened. Then $L_n(x)$ can be written as 
\[
L_n(x)=\sum \limits_{i=1}^m B_iW_i,
\]
where the variables $W_i$ and $B_i$ are independent. Therefore
\begin{align*}
\phi_{L_n(x)}(t)&=\ex e^{itL_n(x)}\\
&=\prod \limits_{i=1}^m [\pr[B_i=0]+\pr[B_i=1]\phi_{W}(t)].
\end{align*}
The probability that the interval $I_i$ has a visit (and then an excursion) was calculated in $\eqref{prob}$, so we write
\begin{align*}
\phi_{L_n(x)}(t)&=\prod \limits_{i=1}^m \left(1-\frac{\alpha x}{2\ln n}+\ord(n^{-1})+\Big(\frac{\alpha x}{2\ln n}+\ord(n^{-1})\Big)\phi_{W}(t)\right)\\
&=\left(1+\frac{\alpha x}{2\ln n}(\phi_{W}(t)-1)+\ord(n^{-1})\right)^{m}\\
&=\exp\Big\{m\ln\left(1+\frac{\alpha x}{2\ln n}(\phi_{W}(t)-1)+\ord(n^{-1})\right)\Big\}.
\end{align*}
As $|\phi_{W}(t)-1|\leq 2$, we can use the asymptotic expansion $\ln(1+x)=x+\ord(x^2)$ and
\begin{align*}
\phi_{L_n(x)}(t)&=\exp\Big\{\frac{mx\alpha }{2\ln n}(\phi_{W}(t)-1)+\ord(m(\ln n)^{-2})\Big\}
\end{align*}
and, since $m=\lfloor\ln n\rfloor$, this becomes 
\begin{align*}
\phi_{L_n(x)}(t)&=\exp\Big\{\frac{\alpha x}{2}(\phi_{W}(t)-1)+\ord((\ln n)^{-1})\Big\}\\
&\rightarrow \exp\Big\{\frac{\alpha x}{2}(\phi_{V}(t)-1)\Big\}.
\end{align*}

Now, using Lemma~\ref{lemmacrw}, we have that $V\sim {\rm Geometric}({(2x)^{-1}})$ with $\pr[V=k]=\frac{1}{2x}(1-\frac{1}{2x})^{k-1}$ and, finally,
\begin{align*}
\phi_{L_n(x)}(t)&\rightarrow \exp\Big\{{\alpha x^2}\frac{(e^{it}-1)}{2x-(2x-1)e^{it}}\Big\}\\
&=\phi_{\ell(x)}(t).
\end{align*}
Then, since $\phi_{\ell(x)}(t)$ is continuous at~$0$, one can use the continuity theorem (cf.\ e.g.\ Theorem~9.5.2 of~\cite{ppath}) and conclude that
\[
L_n(x)\lawlimit \ell(x),
\]
as desired.
\end{proof}

\bibliography{refe}
\end{document}